\documentclass[11pt,letterpaper]{article}
\usepackage{graphicx}
\usepackage{subfigure}
\usepackage{float}
\usepackage{multicol,amsmath,amsfonts,amsthm,color}
\usepackage{lineno}
\usepackage{hyperref}
\hypersetup{urlcolor=blue, colorlinks=true, linkcolor=black,citecolor=black}
\newtheorem{definition}{Definition}
\newtheorem{theorem}{Theorem}
\newtheorem{lemma}{Lemma}
\newtheorem{corollary}{Corollary}
\begin{document}

\begin{center}  {Stability and Hopf bifurcation in a delayed viral infection model with mitosis transmission}
\end{center}

\smallskip

\smallskip
\begin{center}
{\small \textsc{Eric \'Avila--Vales}\footnote{Corresponding author. email: avila@uady.mx}, \textsc{No\'e Chan--Ch\i}\footnote{email: noe.chan@uady.mx},\textsc{Gerardo E. Garc\'ia-Almeida}\footnote{email: galmeida@uady.mx}, \textsc{Cruz Vargas-De-Le\'on}\footnote{email: leoncruz82@yahoo.com.mx}
}
\end{center}
\begin{center} {\small \sl $^{1,2,3}$ Facultad de Matem\'aticas, Universidad Aut\'onoma de Yucat\'an, \\ Anillo Perif\'erico Norte, Tablaje 13615, C.P. 97119, M\'erida, Mexico}\\ {\small \sl $^{4}$Unidad Acad\'emica de Matem\'aticas, Universidad Aut\'onoma de Guerrero, Chilpancingo, Guerrero, M\'exico}
\end{center}

\bigskip

{\small \centerline{\bf Abstract}
\begin{quote}
In this paper we study a model of HCV with mitotic proliferation, a saturation infection rate and a discrete intracellular delay: the delay corresponds to the time between infection of a infected target hepatocytes and production of new HCV particles.  We establish the global stability of the infection--free equilibrium and existence, uniqueness, local and global stabilities of the infected equilibrium, also we establish the occurrence of a Hopf bifurcation. We will determine conditions for the permanence of model, and the length of delay to preserve stability. The unique infected equilibrium is globally-asymptotically stable for a special case, where the hepatotropic virus is non-cytopathic. \\
We present a sensitivity analysis for the basic reproductive number. Numerical simulations are carried out to illustrate the analytical results.
\end{quote}}

\medskip


\bigskip

\section{Introduction}\label{Intro}
The mathematical theory of viral infections has been around during the last decades. This theory has proven to be valuable on the understanding of the dynamics of viral infections and in the evaluation  of effectiveness of antiviral therapy. The most studied viruses in the mathematical theory of viral infections are human immunodeficiency virus (HIV), hepatitis C virus (HCV), human T-lymphotropic virus type 1 (HTLV-1) and hepatitis B virus (HBV).

In the mathematical theory of viral infections the basic models are concentrated on two steps of the lytic cycle: viral entry and release. One of the most well-known models for viral dynamics includes only three state variables, one variable for uninfected target cells, a second variable for infected cells, and a third variable for virions or free virus particles \cite{neumann,Nowakbook,PerelsonHIV}. The basic model  of viral dynamics  that has been formulated  by Neumann  and co--workers \cite{neumann} for HCV dynamics and this has been applied to the analysis  of response to antiviral therapy.  The ordinary differential equations (ODE) model is as follows:
\[ \begin{aligned}
\dot T(t)=& s-dT(t)-(1-\eta )b V(t)T(t), \\
\dot I(t)=& (1-\eta )b V(t)T(t)-\mu I(t), \\
\dot V(t)=& (1-\epsilon)pI(t)-cV(t),
\end{aligned} \]
Here $x(t)$, $y(t)$ and $v(t)$ denote the concentration of uninfected hepatocytes  (or target  hepatocytes), infected hepatocytes and virions or free virus particles, respectively. All parameters are assumed to be positive constants. Here, target cells are generated at a constant rate $s$ and die at rate $d$ per uninfected cell. These cells are infected at rate $b$ per target cell per virion. Infected cells die at rate $\mu$ per cell by cytopathic effects. Because of the viral burden on the virus-infected cells, we assume that $d\leq\mu$. Parameter $p$
represents the average rate at which the  hepatitis C virus are produced by infected cells and $c$ is the clearance rate of virus particles. Parameter $\eta$ is the efficiency of drug therapy in preventing new infections, and $\epsilon$ is the efficiency of drug therapy in inhibiting viral production.

The intracellular viral life-cycle is an important process to consider in mathematical models.  In the mathematical theory of viral infections, the intracellular delays of viral life cycle are modelled by delay differential equations. This class of models are concentrated on two intracellular delays:  The first delay represents the time between viral entry into a uninfected-target cell and the production of new virus (See, \cite{Ruan,Culshaw,Herz,Huangvirus,LiHIV,Zhu}) and the second delay corresponds to the time necessary for a newly produced virus to become infectious virus particles (See, \cite{Huangvirus,Tam,CVDL12,Zhu}).

The first model that included a intracellular delay was developed by Herz and co--workers \cite{Herz}, this model  incorporate the delay between the time a cell is infected and the time it starts producing virus. Subsequently, Li and Shu \cite{MYLi} studied the global stability of the associated equilibrium points  for this delayed system. Tam \cite{Tam} suggested an virus dynamics model with virus production delay. Zhu and Zuo \cite{Zhu} studied the basic model of viral infections with two intracellular delays: cell infection and virus production. Huang, Takeuchi and Ma \cite{Huangvirus} derived a class of within--host virus models with a nonlinear incidence rate and discrete intracellular delays, followed by its global analysis.

The basic model of viral infections \cite{Herz,MYLi,Zhu} assumes a source of uninfected cells but ignores mitotic proliferation of  uninfected cells or infected cells.  Later on,  Dahari and co--workers \cite{dahari1,dahari2} for hepatitis  C viral, extending the basic model \cite{neumann,Nowakbook,PerelsonHIV} include mitotic proliferation terms for both uninfected and infected hepatocytes. Here they assume that the proliferation of cells due to mitotic division obeys a logistic  growth law. The mitotic proliferation of uninfected cells is described by $aT\left( 1-\frac{T(t)+I(t)}{T_{\max}} \right).$ New infectious transmission occurs at a rate $b T(t) V(t)$, while new mitotic transmission occurs at a rate $aI(t)\left( 1-\frac{T(t)+I(t)}{T_{\max}}\right)$.  Both infected and uninfected cells can proliferate with maximum proliferation rate $a$, as long as the total number of cells, $T(t)+I(t)$, is less than $T_{max}$. The ODE model is as follows:
\[\begin{aligned}
\dot T(t)=& s+aT\left( 1-\frac{T(t)+I(t)}{T_{\max}} \right)-dT(t)-(1-\eta)b T(t)V(t),\\
\dot I(t)=& (1-\eta)b T(t)V(t)+aI(t)\left( 1-\frac{T(t)+I(t)}{T_{\max}}\right)-\mu I(t),\\
\dot V(t)=& (1-\epsilon)pI(t)-cV(t).
\end{aligned}\]
In other words, the hepatitis model includes mitotic proliferation
of uninfected  cells, and mitotic transmission of infection
through infected cell division. Also, this model assumed that
infection could occur instantaneously once a virus contacted a
target cell to infect a uninfected target cell. Hu  and
co--workers \cite{Hu} introduced a discrete time delay to the
hepatitis model to describe the time between infection of a
hepatocyte and the emission of viral particles. Subsequently,
Vargas-De-Le\'on \cite{CVDL11} studied the parameter conditions
for global stability of hepatitis model with mitotic proliferation
for both uninfected and infected hepatocytes.  The hepatitis model
given in \cite{dahari1,dahari2} without delay, coincides with the
HIV model studied in \cite{WangHIV} for which the global stability
analysis was completed.

In several of  models with or without delay described above, the
process of cellular infection by free virus particles  are
typically modelled by mass action principle, that is to say, the
infection rate is assumed to occur at a rate proportional to the
product of the abundance of free-virus particles and uninfected
target cells. This principle is insufficient to describe the
cellular infection process in detail, and some nonlinear infection
rates were proposed.  Li and Ma \cite{LiHIV} and Song  and Neumann
\cite{Song}  considered a virus dynamics model with monod
functional response, $bT(t)V(t)/(1+\alpha V(t))$. Regoes, Ebert,
and Bonhoeffer \cite{Regoes} and Song  and Neumann \cite{Song}
considered a virus dynamics model with the nonlinear infection
rates $bT(t)((V(t)/\kappa)^{p})/(1+(V(t)/\kappa)^{p})$ and
$bT(t)(V(t))^{q}/(1+\alpha (V(t))^{p})$ where $p$, $q$, $\kappa>0$
are constants, respectively. Recently Huang, Takeuchi and Ma
\cite{Huangvirus} considered a class of models of viral infections
with an nonlinear infection rate and two discrete intracellular
delays, and assumed that the infection rate is given by a general
nonlinear function of the abundance of free-virus particles and
uninfected target cells, $F(T(t),V(t))$, where the function
$F(T(t),V(t))$ satisfies the concavity with respect to the
abundance of free-virus particles. Such condition is satisfied by
several well known infection rates. The DDE model given in
\cite{Huangvirus} without delay is studied by Korobeinikov in
\cite{Korobeinikov}.

Motivated by the above comments, in the present paper, we proposed a model  more realistic by using an infection rate that saturates and a discrete intracellular delay. We consider the following delay differential equation (DDE) model
\begin{equation} \label{hiv3}\begin{aligned}
\dot T(t)=& s-dT(t)+aT(t)\left( 1-\frac{T(t)+I(t)}{T_{\max}}\right)-\frac{bT(t)V(t)}{1+\alpha V(t)}, \\
\dot I(t)=& \frac{bT(t-\tau)V(t-\tau)}{1+\alpha V(t-\tau)}+aI(t)\left( 1-\frac{T(t)+I(t)}{T_{\max}}\right)-\mu I(t), \\
\dot V(t)=& pI(t)-cV(t).
\end{aligned} \end{equation}

The assumptions are the following. We assume that the contacts between viruses and uninfected target cells are given by an infection rate $bT(t)V(t)/(1+\alpha V(t))$, it is reasonable for us to assume that the infection has a maximal rate of $b/\alpha$. On the other hand, to account for the time between viral entry into an uninfected target cell  and the production of an actively infected target cell, we introduce a  time delay $\tau$ that represents the time from entry to production of new virus.

From the point of view of applications, the study of the asymptotic stability of equilibria, the permanence and the existence of orbit periodic are interesting topics in biological models. Particularly, the qualitative analysis of the models reveals the existence of scenarios possible of a viral infection. A first scenario is that the viral population is eventually totally cleared. Mathematically, this means that the infection--free equilibrium state is asymptotically stable. Biologically, the permanence characterizes that the virus is not cleared. The permanence take care of the second and third scenarios, which will be described below. A second scenario is that the infection becomes established, and that the virus population grows with damped oscillations, or unimodal growth, that is the infected equilibrium state (interior equilibrium with all components positive) is asymptotically stable. A third scenario is that the virus population grows with self-sustained oscillations, is that the interior equilibrium state is unstable.

In this paper, we shall  study  the possible scenarios of virus dynamic model \eqref{hiv3}. The paper is organized as follows. In section 2, we establish the positivity of solutions. In the section 3 we prove the existence of the equilibria, we perform the local stability analysis and the global analysis. The Hopf bifurcation analysis is presented in section 4. Conditions for the permanence are establish in section 5. In section 6, we perform an estimation of the length of delay to preserve stability. The numerical validation is found in section 7, a sensitivity analysis is presented in section 8 and finally we draw a conclusion in section 9.
\section{Positivity of solutions}
We denote the Banach space of continuous functions $\phi : [-\tau ,0]\rightarrow \mathbb{R}$ with norm
\[ ||\phi ||=\sup _{-\tau \leq \theta \leq 0}\left\{ |\phi _1|, |\phi _2|, |\phi _3|,\right\}\]
by $\mathcal{C}$, where $\phi=(\phi _1,\phi _2,\phi _3)$. Further, let
\[ \mathcal{C}_+=\left\lbrace (\phi _1,\phi _2,\phi _3) \in \mathcal{C}: \phi _i\geq 0 \text{ for all }\theta \in [-\tau ,0], i=1,2,3 \right\rbrace .\]
The initial conditions for system \eqref{hiv3} are
\begin{equation} \label{hiv4}
T(\theta )=\phi_1 (\theta)\geq 0,\; I(\theta )=\phi_2 (\theta)\geq 0, \; V(\theta )=\phi_3 (\theta)\geq 0, \; \theta \in [-\tau ,0],
\end{equation}
where $\phi =(\phi_1,\phi_2,\phi_3)$.
\begin{lemma} All solutions of system \eqref{hiv3} with initial conditions \eqref{hiv4} are positive.
\end{lemma}
\begin{proof}
We prove the positivity by contradiction. Suppose $T(t)$ is not always positive. Then, let $t_0>0$ be the first time such that $T(t_0)=0$. From the first equation of \eqref{hiv3} we have $\dot T(t_0)=s>0$. By our assumption  this means $T(t)<0$, for $t\in (t_0-\epsilon ,t_0)$, where $\epsilon$ is an arbitrary small positive constant. Implying that exist $t_0'<t_0$ such that $T(t_0')=0$ this is a contradiction because we take $t_0$ as the first value which $T(t_0)=0$. It follows that $T(t)$ is always positive.

We now show that $I(t)>0$ for all $t>0$. Otherwise, if it is not valid, nothing that $I(0)>0$ and $I(t)>0$, $(-\tau \leq t\leq 0)$, then there exist a $t_1$ such that $I(t_1)=0$. Assume that $t_1$ is the first time such that $I(t)=0$, that is, $t_1=\inf \{ t>0: \; I(t)=0\}$.

Then $t_1>0$, and from system \eqref{hiv3} with \eqref{hiv4}, we get
\[
\dot I(t_1)=\bigg\{ \begin{array}{ll}
\varphi _1(t_1-\tau)\varphi _3(t_1-\tau)>0, & \text{ if }0\leq t_1\leq \tau, \\
T(t_1-\tau)V(t_1-\tau)>0, & \text{ if } t_1>\tau.
\end{array}
\]
Thus, $\dot I(t_1)>0$. Thus, for sufficiently small $\epsilon>0$, $\dot I(t_1-\epsilon)>0$. But by the definition of $t_1$, $\dot I(t_1-\epsilon)\leq 0$. Again this is a contradiction. Therefore, $I(t)>0$ for all $t$.

In the same way, we see that $V(t)$ is always positive. Thus, we can conclude that all solutions of system \eqref{hiv3} and \eqref{hiv4} remain positive for all $t>0$.
\end{proof}
\section{Stability Analysis}
\subsection{Equilibrium}
To obtain the equilibrium points we look for constant solutions for system \eqref{hiv3}, and we obtain $E_1(T_0,0,0)$ and $E_2(T_2,I_2,V_2)$, where
\begin{equation} \label{hiv45} \begin{aligned}
T_0=&\frac{T_{\max}}{2a}\left( a-d+\sqrt{(a-d)^2+\frac{4as}{T_{\max}}}\right),\\
\end{aligned} \end{equation}

The equilibrium point $(T_2,I_2,V_2)$ satisfies
\begin{equation} \label{hcv1}\begin{aligned}
0=& s-dT_2+aT_2\left( 1-\frac{T_2+I_2}{T_{\max}}\right)-\frac{bT_2I_2}{1+\alpha V_2}\\
0=& \frac{bT_2V_2}{1+\alpha V_2}-\mu I_2+aI_2\left( 1-\frac{T_2+I_2}{T_{\max}} \right)\\
0=& pI_2-cV_2.
\end{aligned}\end{equation}

The third equation leads to $V_2=\dfrac{pI_2}{c}$, which allows us to reduce system \eqref{hcv1} to two equations, defining $\tilde \alpha =\dfrac{\alpha p}{c}$ and $\tilde b=\dfrac{bp}{c}$.

We express the first equation of \eqref{hcv1} as the following quadratic equation in $T_2$:
\[
\frac{a}{T_{\max}}T_{2}^{2}+\left(  d-a+\frac{aI_{2}}{T_{\max}}+\frac{bI_{2}}{1+\tilde{\alpha}I_{2}}\right)  T_{2}-s=0.
\]

Note that this quadratic equation has two real roots of opposite sign that depend on $I_{2}$. We are interested in the positive one, which is clearly a function of $I_{2}.$

Defining $T_2=f(I_2)$, we express the first two equations of system \eqref{hcv1} as
\begin{equation} \label{hcv2} \begin{aligned}
0=& s-df(I_2)+af(I_2)\left( 1-\dfrac{f(I_2)+I_2}{T_{\max}}\right)-\frac{\tilde{b}f(I_2)I_2}{1+\tilde \alpha I_2} \\
0=& \frac{\tilde bf(I_2)I_2}{\mu (1+\tilde \alpha I_2)}-I_2+\frac{a}{\mu}I_2\left( 1-\dfrac{f(I_2)+I_2}{T_{\max}}\right).
\end{aligned}\end{equation}
To compute the infection-free equilibrium we assume $I_2=V_2=0$, and we obtain
\[ T_0=f(0)=\dfrac{T_{\max}}{2a}\left( (a-d)+\sqrt{(a-d)^2+\frac{4as}{T_{\max}}}\right).\]
Now if we consider the infected  equilibrium, then $I_2>0$. From the second equation of \eqref{hcv2} we can define
\[F(I_2)=\frac{\tilde b}{\mu}\frac{f(I_2)}{1+\tilde \alpha I_2}+\frac{a}{\mu}\left( 1-\frac{f(I_2)+I_2}{T_{\max}} \right).\]
Considering that $f(0)=T_0$, then
\[ F(0)=\frac{\tilde b}{\mu}T_0+\frac{a}{\mu}\left( 1-\frac{T_0}{T_{\max}}\right)=R_0=\frac{D}{\mu},\]
where $D=\tilde b T_0+a\left( 1-\dfrac{T_0}{T_{\max}}\right)$.

We can rewrite the second equation of \eqref{hcv2} as $F(I_2)=1$.

Geometrically, we can interpret that equation as the intersections of the graph of the function $F\left(  I_{2}\right)$ with the line $F=1$ in the plane determined by $F$ and $I_{2}.$ In order to have a biologically feasible infected equilibrium, the intersections to be considered are those in the first
quadrant of that plane. Now, if the slope of the tangent line along the graph of the function $F\left(  I_{2}\right)  $ is negative and does not approach zero as $I_{2}$ increases beginning at $I_{2}=0,$ as well as $F\left( 0\right)  >1,$ then there is only one of such intersections in the first quadrant, meaning a unique biologically feasible infected equilibrium. Now.

\[ \begin{aligned}
F'(I_2)=& \frac{\tilde b}{\mu }\frac{f'(I_2)}{1+\tilde \alpha }-\frac{\tilde b \tilde \alpha f(I_2)}{\mu (1+\tilde \alpha I_2)^2} -\frac{a}{\mu}\frac{f'(I_2)}{T_{\max}}-\frac{a}{\mu T_{\max}}\\
=& \frac{1}{\mu} \left( \frac{\tilde b}{1+\tilde \alpha I_2}-\frac{a}{T_{\max}}\right)f'(I_2)-\frac{1}{\mu} \left( \frac{\tilde b \tilde \alpha f(I_2)}{(1+\tilde \alpha I_2)^2} +\frac{a}{T_{\max}}\right)
\end{aligned} \]
Note that $F'(I_2)$ depends on $f'(I_2)$. To calculate $f'(I_2)$ we first rewrite the first equation of \eqref{hcv2} as
\[
0=\frac{s}{f(I_2)}-d+a\left( 1-\frac{f(I_2)+I_2}{T_{\max}} \right)-\frac{bI_2}{1+\tilde \alpha I_2}.
\]
Using implicit differentiation we get
\[
f'(I_2)=-\left(
\frac{a}{T_{\max}}+
\frac{\tilde b} {(1+\tilde \alpha I_2)^2}
\right)
\left( \frac{a}{T_{\max}}+\frac{s}{f^2(I_2)} \right) ^{-1}<0
\]
then
\[ \begin{aligned}
F'(I_2)&=\frac{1}{\mu}\left( \left( \frac{a}{T_{\max}}\right) ^2- \left( \frac{\tilde b}{1+\tilde \alpha I_2}\right) ^2\right)\left( \frac{a}{T_{\max}}+\frac{s}{f^2(I_2)}\right)^{-1} -\frac{1}{\mu} \left(  \frac{\tilde b\tilde \alpha f(I_2)}{(1+\tilde \alpha I_2)^2} +\frac{a}{T_{\max}}\right) \\
&< \frac{1}{\mu} \left( \frac{a}{T_{\max}}\right) ^2\left( \frac{a}{T_{\max}}+\frac{s}{f^2(I_2)}\right)^{-1} -\frac{1}{\mu} \left(  \frac{\tilde b\tilde \alpha f(I_2)}{(1+\tilde \alpha I_2)^2} +\frac{a}{T_{\max}}\right) \\
&< \frac{1}{\mu} \left( \frac{a}{T_{\max}}\right) ^2\left( \frac{a}{T_{\max}}\right)^{-1} -\frac{1}{\mu} \left(  \frac{\tilde b\tilde \alpha f(I_2)}{(1+\tilde \alpha I_2)^2} +\frac{a}{T_{\max}}\right) \\
&= -\frac{1}{\mu} \frac{\tilde b\tilde \alpha f(I_2)}{(1+\tilde \alpha I_2)^2}<0.
\end{aligned} \]
Note that the minimum value that the denominator $\left(  \dfrac{a}{T_{\max}}+\dfrac{s}{f^{2}(I_{2})}\right)$ can take is $$\left(\dfrac{a}{T_{\max}
}+\dfrac{s}{f^{2}(0)}\right)  =\left(  \dfrac{a}{T_{\max}}+\dfrac{s}{T_{0}^{2}}\right)>\dfrac{a}{T_{\max}}.$$ Hence $F^\prime(I_{2})$ is negative and does not approach zero as $I_{2}$ increases. If, in addition $R_{0}>1,$ then $F\left(  0\right)  >1$ and the slope of the tangent line along the graph of the function $F\left(  I_{2}\right)  $ is negative and does not approach zero as $I_{2}$ increases beginning at $I_{2}=0.$ It follows that the infected equilibrium exists and it is unique if $R_{0}>1$. Where
\begin{equation} \label{hivr0}
R_0=\frac{1}{\mu} \left[ \frac{bpT_0}{c}+a\left( 1-\frac{T_0}{T_{\max}} \right) \right]
\end{equation}
as in \cite{debroy}
\subsection{Local Analysis}
In this section we study the local stability of the infection-free equilibrium $E_1$ and the infected equilibrium $E_2$. \\
The characteristic equation of system \eqref{hiv3} at the infection-free equilibrium is of the form
\begin{equation} \label{hiv5}
\left( \lambda+d-a+\frac{2aT_0}{T_{\max}}\right) \left( \lambda ^2+\left(c+\mu -a+\frac{aT_0}{T_{\max}}\right)\lambda +c(\mu -a)+\frac{acT_0}{T_{\max}} -bpT_0{\rm e}^{-\lambda \tau}\right)
\end{equation}
we also consider that $E_1$ satisfies system \eqref{hiv3}, so $a\left( 1-\dfrac{T_0}{T_{\max}}\right) =d-\dfrac{s}{T_0}$. And using the previous fact we can rewrite the factors of the characteristic equation \eqref{hiv5} as
\[\lambda =a-d-\frac{2aT_0}{T_{\max}} = -\left( \frac{s}{T_0}+\frac{aT_0}{T_{\max}} \right)\]
which have a negative eigenvalue, and the others eigenvalues are given by
\begin{equation} \label{hiv6}
\lambda ^2+a_1\lambda +a_0+b_0{\rm e}^{-\lambda \tau }=0
\end{equation}
where
\[\begin{aligned}
a_1=& c+\mu -d+\frac{s}{T_0},\\
a_0=& c(\mu -d)+\frac{cs}{T_0},\\
b_0=& -bpT_0
\end{aligned}\]
Let
\[f(\lambda )=\lambda ^2+a_1\lambda +a_0+b_0{\rm e}^{-\lambda \tau }\]
note that
\[ \begin{aligned}
a_0+b_0=&-\frac{1}{c}\left[ \frac{bpT_0}{c}+d-\frac{s}{T_0}-\mu\right]=-\frac{1}{c}\left[ \frac{bpT_0}{c}+a\left( 1-\frac{T_0}{T_{\max}}\right)-\mu\right] \\
=&-\frac{\mu }{c}(R_0-1)
\end{aligned} \]
if $R_0>1$, and we consider $\lambda$ real we have that
\[f(0)=a_0+b_0<0 \text{ and } \lim _{\lambda \rightarrow \infty}f(\lambda) =\infty.\]
Hence exist a positive root $\lambda ^*$ of $f(\lambda)$, therefore the characteristic equation  \eqref{hiv5} has a positive root and the infection-free equilibrium is unstable.

When $\tau =0$ the equation \eqref{hiv6} becomes
\[ \lambda ^2+a_1\lambda +a_0+b_0=0,\]
with $a_1>0$ and, if $R_0<1$, we have $a_0+b_0>0$. Hence, the equilibrium $E_1$ is locally asymptotically stable when $\tau =0$.

If $\lambda =i\omega$ ($\omega >0$) is a solution of \eqref{hiv6} then separating in real and imaginary parts we obtain the system
\[ \begin{aligned}
a_0-\omega ^2&=-b_0\cos (\omega \tau),\\
a_1\omega &=b_0\sin (\omega \tau),
\end{aligned}\]
squaring and adding the last two equations and after simplifications we get
\begin{equation} \label{hiv7}
\omega ^4+(a_1^2-2a_0)\omega ^2+a_0^2-b_0^2=0,
\end{equation}
if $R_0<1$ the equation \eqref{hiv7} has no positive roots. Noting that the equilibrium $E_1$ is locally asymptotically stable when $\tau =0$, by the theory on characteristic equations of delay differential equations from Kuang \cite{kuang}, we see that if $R_0<1$, $E_1$ is locally asymptotically stable.

The characteristic equation of system \eqref{hiv3} on the infected equilibrium $E_2$ is given by the following determinant
\[
\det \left( \begin{matrix}
a-d-\frac{a}{T_{\max}}[2T_2+I_2] -\frac{bV_2}{1+\alpha V_2}-\lambda & -\frac{a}{T_{\max}}T_2 & -\frac{bT_2}{(1+\alpha V_2)^2} \\
\frac{bV_2}{1+\alpha V_2}e^{-\lambda \tau}-\frac{a}{T_{\max}}I_2 & a-\mu -\frac{a}{T_{\max}}[T_2+2I_2]-\lambda & \frac{bT_2}{(1+\alpha V_2)^2}e^{-\lambda \tau} \\
0 & p & -c-\lambda
\end{matrix} \right)
\]
that can be rewritten as
\[
\det \left( \begin{matrix}
-\left( \rho + \frac{a}{T_{\max}}T_2+ \frac{bV_2}{1+\alpha V_2}\right)-\lambda & -\frac{a}{T_{\max}}T_2 & -\frac{bT_2}{(1+\alpha V_2)^2} \\
\frac{bV_2}{1+\alpha V_2}e^{-\lambda \tau}-\frac{a}{T_{\max}}I_2 & -\left( \frac{bT_2V_2}{(1+\alpha V_2)I_2}+\frac{a}{T_{\max}}I_2\right)-\lambda & \frac{bT_2}{(1+\alpha V_2)^2}e^{-\lambda \tau} \\
0 & p & -\frac{pI_2}{V_2}-\lambda
\end{matrix} \right)
\]
and the characteristic equations is of the form
\begin{equation} \label{hiv8}
\lambda ^3+a_2\lambda ^2+a_1\lambda +a_0 + \left(b_1\lambda +b_0 \right){\rm e}^{-\lambda \tau}=0
\end{equation}
where
\[ \begin{aligned}
a_2=& \rho +\frac{aT_2}{T_{\max}}+\frac{bV_2}{1+\alpha V_2}+\frac{bT_2V_2}{(1+\alpha V_2)I_2}+\frac{aI_2}{T_{\max}}+\frac{pI_2}{V_2}>0\\
a_1=& \left( \rho +\frac{aT_2}{T_{\max}}+\frac{bV_2}{1+\alpha V_2}\right)\left( \frac{bT_2V_2}{(1+\alpha V_2)}+\frac{aI_2}{T_{\max}}\right)+\frac{pI_2}{V_2}\left( \rho +\frac{aT_2}{T_{\max}}+\frac{bV_2}{1+\alpha V_2}\right)\\
&+\frac{pI_2}{V_2}\left( \frac{bT_2V_2}{(1+\alpha V_2)I_2}+\frac{aI_2}{T_{\max}}\right)-\frac{aT_2}{T_{\max}}\frac{aI_2}{T_{\max}}>0\\
a_0=& \left( \rho +\frac{aT_2}{T_{\max}}+\frac{bV_2}{1+\alpha V_2}\right) \left( \frac{bT_2V_2}{(1+\alpha V_2)I_2} \right)\frac{pI_2}{V_2}-\frac{pbT_2}{(1+\alpha V_2)^2}\frac{aI_2}{T_{\max}}-\frac{aT_2}{T_{\max}}\frac{pI_2}{V_2}\frac{aI_2}{T_{\max}}\\
b_1=& \frac{aT_2}{T_{\max}}\frac{bV_2}{1+\alpha V_2}-\frac{pbT_2}{(1+\alpha V_2)^2}\\
b_0=& \frac{aT_2}{T_{\max}}\frac{pI_2}{V_2}\frac{bV_2}{1+\alpha V_2} +\frac{pbT_2}{(1+\alpha V_2)^2} \frac{bV_2}{1+\alpha V_2}-\frac{pbT_2}{(1+\alpha V_2)}\left( \rho +\frac{aT_2}{T_{\max}}+\frac{bV_2}{1+\alpha V_2}\right)
\end{aligned} \]
with $T_2$, $I_2$, $V_2$ defined in \eqref{hiv45}. \\
When $\tau =0$ equation \eqref{hiv8} becomes
\[ \lambda ^3+a_2\lambda ^2 +(a_1+b_1)\lambda +(a_0+b_0)=0\]
by the Routh--Hurwitz criterion the conditions for ${\rm Re }\lambda<0$ are $a_2>0$, $a_0+b_0>0$, $a_2(a_1+b_1) - (a_0+b_0)>0$, in our case $a_2>0$ and
\[ \begin{aligned}
a_0+b_0=&\frac{pbT_2}{1+\alpha V_2}\left( \rho +\frac{aT_2}{T_{\max}}+\frac{bV_2}{1+\alpha V_2}\right)\left( 1-\frac{1}{1+\alpha V_2}\right) +\frac{aI_2}{T_{\max}}\frac{pI_2}{V_2}\left( \rho +\frac{aT_2}{T_{\max}}+ \frac{bV_2}{1+\alpha V_2}\right) \\
&+\frac{pbT_2}{(1+\alpha V_2)^2}\left( \frac{bV_2}{1+\alpha V_2}-\frac{aI_2}{T_{\max}}\right) +\frac{aT_2}{T_{\max}}\frac{pI_2}{V_2}\left( \frac{bV_2}{1+\alpha V_2}-\frac{aI_2}{T_{\max}}\right)
\end{aligned}\]
so we need
\begin{itemize}
\item[(H1)] $\quad a_0+b_0>0, \quad a_2(a_1+b_1)-(a_0+b_0)>0$
\end{itemize}
If $\tau =0$, by the Routh--Hurwitz criterion, we have the following theorem
\begin{theorem}
If $(H1)$ is satisfied, then the infected equilibrium $E_2(T_2,I_2,V_2)$ is locally asymptotically stable.
\end{theorem}
Now we analyze if it is possible to have a complex root with positive real part for the case $\tau >0$, assuming (H1) satisfies, note that $\lambda =0$ is not a root of \eqref{hiv8} because $a_0+b_0>0$. Now suppose that $\lambda =i\omega$, with $\omega >0$, is a root of \eqref{hiv8} so the next equation must be satisfied by $\omega$
\[ -\omega ^3i-a_2\omega ^2+a_1\omega i+a_0+(b_1\omega i+b_0)(\cos (\omega \tau)+\sin (\omega \tau)i)=0.\]
Separating again the real and imaginary parts, we have the following system
\begin{equation} \begin{aligned} \label{hiv10}
a_2^2\omega ^2+a_0&=-b_0\cos (\omega \tau)-b_1\omega \sin (\omega \tau),\\
-\omega ^3+a_1\omega &=-b_1\cos (\omega \tau)+b_0\sin (\omega \tau).
\end{aligned} \end{equation}
Now, we square both sides of each equation above and add the resulting equations, to obtain the following sixth degree equation for $\omega $
\begin{equation} \label{hiv11}
\omega ^6+(a_2-2a_1)\omega ^4+(a_1^2-2a_2a_0-b_1^2)\omega ^2+a_0^2-b_0^2=0.
\end{equation}
Let
\[ z=\omega ^2, \quad A=a_2-2a_1, \quad B=a_1^2-2a_2a_0-b_1^2, \quad C=a_0^2-b_0^2\]
then equation \eqref{hiv11} becomes the third order equation in $z$
\begin{equation} \label{hiv12}
z^3+Az^2+Bz+C=0.
\end{equation}

Suppose that \eqref{hiv12} at least a positive root, let $z_0$ the small value for this roots. Then equation \eqref{hiv11} has the root $\omega _0=\sqrt{z_0}$ then form \eqref{hiv10} obtain the value of $\tau$ associated with this $\omega _0$ such that $\lambda =\omega i$ is an purely imaginary root of  \eqref{hiv8}
\[
\tau _0=\frac{1}{\omega _0}\arccos \left[ \frac{b_0(a_2\omega_0 ^2-a_0)+b_1\omega_0 (\omega_0 ^3-a_1\omega_0 )}{b_0^2+b_1^2\omega_0^ 2}\right]
\]
Then we have the following result, from lemma 2.1 from Ruan \cite{ruan}
\begin{theorem}
Suppose (H1) hold.
\begin{itemize}
\item[(a)] If $C\geq 0$ and $\Delta =A^2-3B<0$, then all roots of equation \eqref{hiv8} have negative real parts for all $\tau \geq 0$, then the infected equilibrium $E_2$ is locally asymptotically stable.
\item[(b)] If $C<0$ or $C\geq 0$, $z_1>0$ and $z_1^3+Az_1^2+Bz_1+C\leq 0$, the all roots of equation \eqref{hiv8} have negative real parts when $\tau \in [0,\tau _0)$, then the infected equilibrium $E_2$ is locally asymptotically stable in $[0,\tau _0)$.
\end{itemize}
\end{theorem}
\subsection{Global Analysis}
In this section we study the global stability of the equilibria, the method to prove is to construct a Lyapunov functional
\begin{theorem}
The infection-free equilibrium $E_1(T_0,0,0)$, with $T_0$ defined in \eqref{hiv45}, of system \eqref{hiv3} is globally asymptotically stable if $R_0\leq 1$.
\end{theorem}
\begin{proof}
Define the Lyapunov functional
\[ U(t)=\int _{T_0}^T \frac{\sigma -T_0}{\sigma}d\sigma +I+\frac{bT_0}{c}V+b\int _0^\tau \frac{T(t-\omega)V(t-\omega)}{1+\alpha V(t-\omega)}d\omega\]
$U$ is defined and is continuous for any positive solution $(T(t),I(t),V(t))$ of system \eqref{hiv3} and $U=0$ at $E_1(T_0,0,0)$. And calculating the derivative of $U(t)$ along positive solutions of \eqref{hiv3}, it follows that
\[ \begin{aligned}
\frac{dU}{dt} =& \frac{T-T_0}{T}\dot T(t)+\dot I(t)+\frac{bT_0}{c}\dot V-b\int _0^\tau \frac{d}{d\omega}\frac{T(t-\omega)V(t-\omega)}{1+\alpha V(t-\omega)}d\omega \\
=& \frac{T-T_0}{T}\left( s-dT+aT\left( 1-\frac{T+I}{T_{\max}}\right)-\frac{bTV}{1+\alpha V}  \right)+ \frac{bT(t-\tau)V(t-\tau)}{1+\alpha V(t-\tau )}-\mu I\\
&+aI\left( 1-\frac{I+T}{T_{\max}} \right)
+\frac{bT_0}{c}(pI-cV)-\frac{bT(t-\tau)V(t-\tau)}{1+\alpha V(t-\tau )}+\frac{bTV}{1+\alpha V}\\
=&(T-T_0)\left( \frac{s}{T}-d+a\left( 1-\frac{T+I}{T_{\max}} \right)-\frac{bV}{1+\alpha V} \right) -\mu I + aI\left( 1-\frac{I+T}{T_{\max}} \right)\\
&+ \frac{bT_0}{c}(pI-cV) +\frac{bTV}{1+\alpha V}
\end{aligned} \]
using $a-d=\dfrac{aT_0}{T_{\max}}-\dfrac{s}{T_0}$ and simplifying, we get
\[ \begin{aligned}
\frac{dU}{dt} =& (T-T_0)\left( -s\left( \frac{T-T_0}{TT_0} \right) -\frac{a}{T_{\max}} (T-T_0)-\frac{aI}{T_{\max}}-\frac{bV}{1+\alpha V}\right) -\mu I + aI\left( 1-\frac{I+T}{T_{\max}} \right) \\&+ \frac{bT_0}{c}(pI-cV) +\frac{bTV}{1+\alpha V}\\
=& -s\frac{(T-T_0)^2}{TT_0}-\frac{a}{T_{\max}}(T-T_0)^2 -\frac{a}{T_{\max}}I(T-T_0)-\frac{b(T-T_0)V}{1+\alpha V}-\mu I+aI-\frac{a}{T_{\max}}IT\\
&-\frac{a}{T_{\max}}I^2 +\frac{bpT_0}{c}I-bT_0V+\frac{bTV}{1+\alpha V}\\
\leq & -s\frac{(T-T_0)^2}{TT_0}-\frac{a}{T_{\max}}(T-T_0)^2 -\frac{a}{T_{\max}}I(T-T_0)-\frac{b(T-T_0)V}{1+\alpha V}-\mu I+aI-\frac{a}{T_{\max}}IT\\
&-\frac{a}{T_{\max}}I^2 +\frac{bpT_0}{c}I-\frac{bT_0V}{1+\alpha V}+\frac{bTV}{1+\alpha V} +\frac{a}{T_{\max}}IT_0-\frac{a}{T_{\max}}IT_0\\
= & -s\frac{(T-T_0)^2}{TT_0}-\frac{a}{T_{\max}}(T-T_0)^2 -\frac{2a}{T_{\max}}I(T-T_0)-\mu I+aI-\frac{a}{T_{\max}}I^2 +\frac{bpT_0}{c}I \\
&-\frac{a}{T_{\max}}IT_0\\
=&-s\frac{(T-T_0)^2}{TT_0}-\frac{a}{T_{\max}}[(T-T_0)+I]^2+ I\left( a \left( 1-\frac{T_0}{T_{\max}}\right) +\frac{bT_op}{c}-\mu \right)
\end{aligned} \]
note that
\[ a \left( 1-\frac{T_0}{T_{\max}}\right) +\frac{bT_op}{c}-\mu=\mu (R_0-1)\leq 0,\]
since $R_0\leq 1$, therefore
\[
\dot L(t)\leq -\left( s\frac{(T-T_0)^2}{TT_0}+ \frac{a}{T_{\max}}[(T-T_0)^2+I]+\mu (1-R_0)I \right)<0
\]
If $R_0\leq 1$ then $\dfrac{dU}{dt}\leq 0$ from corollary 5.2 in \cite{kuang}, $E_1$ is globally asymptotically stable. Also, for $R_0=1$, $\dfrac{dU}{dt}(t)=0$ if and only if $T(t)=T_0$ and $I(t)=0$ while in the case $R_0<1$, $\dfrac{dU}{dt}(t)=0$ if and only if $T(t)=T_0$ and $I(t)=0$. Therefore, the largest invariant set in $\left\{ (T(t),I(t),V(t)):\dfrac{dU}{dt}=0\right\}$ when $R_0\leq 1$ is ${E_1(T_0,0,0)}$.  By the classical Lyapunov-LaSalle invariance principle (theorem 5.3 in \cite{kuang}), $E_1$ is globally asymptotically stable.
\end{proof}

In the following, we consider the global asymptotic stability of a unique infected  equilibrium $E_2$. We construct an Lyapunov functional for infected equilibrium, using suitable combinations of the Lyapunov functions given in \cite{Korobeinikov} and the Volterra--type functionals  \cite{McCluskey,CVDL11}.
\begin{theorem} \label{GASinfected}
If $R_0>1$  and $a\leq d+\frac{a}{T_{\max}}[T_2+I_2]$, then the unique infected equilibrium $E_2$ of (\ref{hiv3}) is globally asymptotically stable for any $\tau\geq0$.
\end{theorem}

\begin{proof}
Define a Lyapunov functional for $E_2$,
\begin{align*}
L(t)&=\widetilde{L}(t)+ \frac{bT_2V_2}{1+\alpha V_2} L_+(t),
\end{align*}
where
\begin{align*}
\widetilde{L}&=\int_{T_2}^{T}\frac{(\sigma -T_2)}{\sigma} d\sigma +\int_{I_2}^{I} \frac{(\sigma -I_2)}{\sigma} d\sigma +\frac{bT_2V_2}{pI_2(1+\alpha V_2)} \int_{V_2}^{V}\left(1-\frac{V_2(1+\alpha \sigma)}{\sigma(1+\alpha V_2)}\right)d\sigma,
\end{align*}
and
\begin{align*}
L_+&=\int^{\tau}_{0}\left(\frac{T(t-\omega)V(t-\omega)(1+\alpha V_2)}{T_2V_2(1+\alpha V(t-\omega))} -1- \ln\frac{T(t-\omega)V(t-\omega)(1+\alpha V_2)}{T_2V_2 (1+\alpha V(t-\omega))}\right)d\omega.
\end{align*}
At infected equilibrium, we have
\begin{eqnarray}
a-d&=&-\frac{s}{T_2}+\frac{bV_2}{1+\alpha V_2}+ \frac{ a}{T_{\max}}(T_2+I_2),\label{identities3}\\
a-\mu &=&-\frac{bT_2V_2}{I_2(1+\alpha V_2)}+\frac{a}{ T_{\max}}(T_2+I_2), \label{identities4}\\
c&=&p\frac{I_2}{V_2}. \label{identities5}
\end{eqnarray}
The derivative of $\widetilde{L}$ with respect to $t$ along the solutions of (\ref{hiv3}), we get
\begin{align*}
\frac{d\widetilde{L}}{dt}&=\frac{(T-T_2)}{T}\dot T +\frac{(I-I_2)}{I} \dot I+\frac{bT_2V_2}{pI_2(1+\alpha V_2)}\left(1-\frac{V_2(1+\alpha V)}{V(1+\alpha V_2)}\right)\dot V,\\
&=(T -T_2)\left(\frac{s}{T}-\frac{a}{T_{\max}}(T+I)- \frac{bV}{1+\alpha V}+a-d\right)\\
&+(I-I_2)\left(\frac{bT(t-\tau)V(t-\tau)}{
I(1+\alpha V(t-\tau))}-\frac{a}{T_{\max}}(T+I)+a-\mu  \right)\\
&+\frac{bT_2V_2}{pI_2(1+\alpha V_2)}\left(1 -\frac{V_2 (1+\alpha V)}{V(1+\alpha V_2)}\right)(pI-cV).\end{align*}
Using (\ref{identities3})--(\ref{identities4}) and
(\ref{identities5}), we get

\begin{align*}
\frac{d\widetilde{L}}{dt}&=(T-T_2)\left(-s\frac{(T
-T_2)}{TT_2}-\frac{a}{T_{\max}}[(T-T_2)+(I-I_2)]-b
\left(\frac{V}{1+\alpha V}-\frac{V_2}{1+\alpha V_2}\right) \right)\\
&+(I-I_2)\left( b\left(\frac{T(t-\tau)V(t-\tau)}{ I(1+\alpha V(t-\tau))} -\frac{T_2V_2}{I_2(1+\alpha
V_2)}\right)-\frac{a}{T_{\max}}[(T-T_2)+(I -I_2)] \right)\\
&+\frac{bT_2V_2}{pI_2(1+\alpha V_2)}\left(1 - \frac{V_2(1+\alpha V)}{V(1+\alpha V_2)}\right)\left(p I
-p I_2\frac{V}{V_2}\right).
\end{align*}

Cancelling identical terms with opposite signs and collecting terms, yields

\begin{align*}
\frac{d\widetilde{L}}{dt}&=-s\frac{(T-T_2)^{2}}{TT_2}-\frac{a}{T_{\max}}[(T-T_2)+(I-I_2)]^{2}\\
&+ \frac{bT_2V_2}{1+\alpha V_2} \left(-\frac{TV(1+ \alpha V_2)}{T_2V_2(1+\alpha V)}+ \frac{T(t-\tau) V(t-\tau)(1+\alpha V_2)}{T_2V_2(1+\alpha V(t-\tau))} \right)\\
&+ \frac{bT_2V_2}{1+\alpha V_2} \left(-\frac{IV_2 (1+\alpha V)}{I_2V(1+\alpha V_2)}- \frac{T(t-\tau) I_2V(t-\tau)(1+\alpha V_2)}{T_2IV_2(1+\alpha V(t-\tau))}\right)\\
&+ \frac{bT_2V_2}{1+\alpha V_2} \left(\frac{T}{T_2}+\frac{V(1+\alpha V_2)}{V_2(1+\alpha
V)}-\frac{V}{V_2}+\frac{(1+\alpha V)}{(1+\alpha
V_2)}\right).\end{align*}
We can rewrite $\dfrac{d\widetilde{L}}{dt}$ as
\begin{align*}
\frac{d\widetilde{L}}{dt}&=-s\frac{\left( T
-T_2\right)^{2}}{TT_2}-\frac{a}{T_{\max}
}[(T-T_2)+(I -I_2)]^{2}\\
&+\frac{bT_2V_2}{1+\alpha V_2} \left(-\frac{TV(1+\alpha V_2)}{T_2V_2(1+\alpha V)}+\frac{T(t-\tau)V(t-\tau)(1+\alpha V_2)}{T_2V_2(1+\alpha V(t-\tau))}\right)\\
&+ \frac{bT_2V_2}{1+\alpha V_2}\left(3-\frac{T_2}{T} -\frac{IV_2(1+\alpha V)}{I_2V(1+\alpha V_2)}- \frac{T(t-\tau)I_2V(t-\tau)(1+\alpha V_2)}{T_2IV_2 (1+\alpha V(t-\tau))}\right)\\
&+\frac{bT_2V_2}{1+\alpha V_2} \left(\frac{V(1+\alpha V_2)}{V_2(1+\alpha V)}-\frac{V} {V_2}+\frac{(1+\alpha V)}{(1+\alpha V_2)}-1\right)+\frac{bT_2V_2}{1+\alpha V_2} \left(\frac{T_2}{T}+\frac{T}{T_2}-2\right),\end{align*}
replacing the term $\dfrac{T}{T_2}+\dfrac{T_2}{T}-2$ by
$\dfrac{(T-T_2)^2}{T_2}$,
\begin{align*}
\frac{d\widetilde{L}}{dt}&=-\left(s- \frac{bT_2V_2}{1+\alpha V_2}\right)\frac{(T-T_2)^{2}}{TT_2}-\frac{a}{  T_{\max}}[(T-T_2)+(I-I_2)]^2\\
&+ \frac{bT_2V_2}{1+\alpha V_2}
\left(-\frac{TV(1+ \alpha V_2)}{T_2V_2(1+\alpha V)} +\frac{T(t-\tau) V(t-\tau)(1+\alpha V_2)}{T_2V_2(1+\alpha V(t-\tau))} \right)\\
&+ \frac{bT_2V_2}{1+\alpha V_2} \left(3-\frac{T_2}{T}-\frac{IV_2(1+\alpha V)}{I_2V(1+\alpha V_2)} -\frac{T(t-\tau)I_2V(t-\tau)(1+\alpha V_2)}{T_2IV_2(1+\alpha V(t-\tau))}\right)\\
&+ \frac{bT_2V_2}{1+\alpha V_2} \left(1 -\frac{V_2(1 +\alpha V)}{V(1+\alpha V_2)} \right)\left(\frac{V(1+ \alpha V_2)}{V_2(1+\alpha V)}-\frac{V}{V_2}\right).
\end{align*}

Using $s-\dfrac{bT_2V_2}{1+\alpha V_2}=(
d-a)T_2+\frac{aT_2}{T_{\max}}[T_2+I_2]$, we get

\begin{align*}
\frac{d\widetilde{L}}{dt}&=-\left(d-a+\frac{a}{
T_{\max}}[T_2+I_2]\right)\frac{\left(T
-T_2\right)^2}{T}-\frac{a}{T_{\max}
}[(T-T_2)+(I-I_2)]^2\\
&+\frac{bT_2V_2}{1+\alpha V_2} \left(-\frac{TV(1+\alpha V_2)}{T_2V_2(1+\alpha V)}+\frac{T(t-\tau)V(t-\tau)(1+\alpha V_2)}{T_2V_2(1+\alpha V(t-\tau))}\right)\\
&+ \frac{bT_2V_2}{1+\alpha V_2} \left(3-\frac{T_2}{T}-\frac{IV_2(1+\alpha V)}{I_2V(1+\alpha V_2)}-\frac{T(t-\tau)I_2V(t-\tau)(1+\alpha V_2)}{T_2IV_2(1+\alpha V(t-\tau))}\right)\\
&-\alpha bT_2\frac{\left(V-V_2\right)^2}{(1+\alpha
V)(1+\alpha V_2)^2}.\end{align*}
It is easy to see that
\begin{align*}
\frac{dL_+}{dt}&=\frac{d}{dt}\int^{\tau}_{0} \left(\frac{T(t-\omega)V(t-\omega)(1+\alpha V_2)}{T_2V_2(1+\alpha V(t-\omega))}-1-\ln\frac{T(t-\omega) V(t-\omega)(1+\alpha V_2)}{T_2V_2(1+\alpha V(t-\omega))}\right)d\omega,\\
&=\int^{\tau}_{0}\frac{d}{dt}\left(\frac{T(t-\omega) V(t-\omega)(1+\alpha V_2)}{T_2V_2(1+\alpha V(t-\omega))}-1-\ln\frac{T(t-\omega)V(t-\omega)(1+\alpha V_2)}{T_2V_2(1+\alpha V(t-\omega))}\right) d\omega,\\
&=-\int^{\tau}_{0}\frac{d}{d\omega}\left(\frac {T(t-\omega)V(t-\omega)(1+\alpha V_2)}{T_2V_2(1+\alpha V(t-\omega))}-1- \ln \frac {T(t-\omega)V(t-\omega)(1+\alpha V_2)}{T_2V_2(1+\alpha V(t-\omega))} \right)d\omega,\\
&=-\left[\frac{T(t-\omega)V(t-\omega)(1+\alpha V_2)}{T_2V_2(1+\alpha V(t-\omega))}-1-\ln\frac{T(t-\omega) V(t-\omega)(1+\alpha V_2)}{T_2V_2(1+\alpha V(t-\omega))}\right]^{\tau}_{\omega=0},\\
&=-\frac{T(t-\tau)V(t-\tau)(1+\alpha V_2)}{T_2V_2 (1+\alpha V(t-\tau))}+\frac{TV(1+\alpha V_2)}{T_2V_2(1+\alpha V)}+\ln\frac{T(t-\tau)V(t-\tau)(1+\alpha V_2)}{T_2V_2(1+\alpha V(t-\tau))}\\
&+\ln\frac{T_2V_2(1+\alpha V)}{TV(1+\alpha V_2)},\\
&=-\frac{T(t-\tau)V(t-\tau)(1+\alpha
V_2)}{T_2V_2(1+\alpha V(t-\tau))}+\frac{TV(1+\alpha
V_2)}{T_2V_2(1+\alpha V)}+\ln\frac{T(t-\tau)I_2 V(t-\tau)(1+\alpha V_2)}{T_2IV_2(1+\alpha V(t-\tau))}\\
&+\ln\frac{T_2}{T}+\ln\frac{IV_2(1+\alpha
V)}{I_2V(1+\alpha V_2)}.
\end{align*}
Since
\begin{align*}
\frac{dL}{dt}&=\frac{d\widetilde{L}}{dt}+ \frac{bT_2V_2}{1+\alpha V_2} \frac{dL_+}{dt},
\end{align*}
we obtain
\begin{align*}
\frac{dL}{dt}&=-\left(d-a+\frac{a}{T_{\max}
}[T_2+I_2]\right)\frac{(T-T_2)^2}{T}-\frac{a}{  T_{\max}}[(T-T_2)+(I-I_2)]^2\\
&-\frac{bT_2V_2}{1+\alpha V_2} \left(\frac{T_2}{T}-1- \ln\frac{T_2}{T}\right)- \frac{bT_2V_2}{1+\alpha V_2}
\left(\frac{IV_2(1+\alpha V)}{I_2V(1+\alpha V_2)} -1-\ln\frac{IV_2(1+\alpha V)}{I_2V(1+\alpha V_2)}\right)\\
&- \frac{bT_2V_2}{1+\alpha V_2} \left(\frac{T(t-\tau) I_2V(t-\tau)(1+\alpha V_2)}{T_2IV_2(1+\alpha V(t-\tau))}-1-\ln\frac{T(t-\tau)I_2V(t-\tau)(1+\alpha V_2)}{T_2IV_2(1+\alpha V(t-\tau))}\right)\\
&-\alpha bT_2\frac{\left(V-V_2\right)^2}{(1+\alpha
V)(1+\alpha V_2)^2}.\end{align*}

\vglue.5cm
\noindent Thus, $a\leq d  +\dfrac{a}{T_{\max}}[T_2+I_2]$ implies that $dL/dt\leq0$. By Corollary 5.2 in \cite{kuang}, solutions limit to $\mathbb{M}$, the largest invariant subset of $\left\{dL/dt=0\right\}$. Furthermore, $dL/dt=0$ if and only if $T(t)=T(t-\tau)=T_2$, $V(t)=V(t-\tau)=V_2$ and
$I(t)=I_2$. Therefore the largest compact invariant set in $\mathbb{M}$ is the singleton $\{E_2\}$, where $E_2$ is the infected equilibrium. This shows that
$\lim_{t\rightarrow\infty} (T(t),I(t),V(t))=(T_2,I_2,V_2)$.
By the classical Lyapunov-LaSalle invariance principle (Theorem 5.3 in \cite{kuang}), if $a\leq d +\dfrac{a}{ T_{\max}}[T_2+I_2]$ then $E_2$ is globally asymptotically stable. This proves Theorem \ref{GASinfected}.\end{proof}

\noindent As is well known, the Lyapunov functions are never unique. We constructed a Volterra--type Lyapunov functional for the  infected equilibrium to prove Theorem \ref{GASinfected}

\begin{align*}
L(t)&=\widehat{L}(t)+\frac{bT_2V_2}{1+\alpha V_2}L_+(t), \end{align*}
where
\begin{align*}
\widehat{L}&=\int_{T_2}^T\frac{\left( \sigma -T_2\right)}{\sigma} d\sigma +\int_{I_2}^I\frac{\left( \sigma -I_2\right)}{\sigma}d\sigma +\frac{bT_2V_2}{p I_2(1+\alpha V_2)}\int_{V_2}^V\left(1 -\frac{V_2}{\sigma}\right)d\sigma.
\end{align*}
The time derivative of $L(t)$ computed along solutions of \eqref{hiv3}, is given by the expression
\begin{align*}
\frac{dL}{dt}&=-\left(d-a+\frac{a}{T_{\max}
}[T_2+I_2]\right)\frac{(T-T_2)^2}{T}-\frac{a}{T_{\max}}[(T-T_2)+(I-I_2)]^{2}\\
&- \frac{bT_2V_2}{1+\alpha V_2} \left(\frac{T_2}{T}-1-\ln\frac{T_2}{T}\right)- \frac{bT_2V_2}{1+\alpha V_2}\left(\frac{IV_2}{I_2V}-1-\ln\frac{IV_2}{I_2V} \right)\\
&- \frac{bT_2V_2}{1+\alpha V_2} \left(\frac{T(t-\tau) I_2V(t-\tau)(1+\alpha V_2)}{T_2IV_2(1+\alpha V(t-\tau))}-1-\ln\frac{T(t-\tau)I_2V(t-\tau)(1+\alpha V_2)}{T_2IV_2(1+\alpha V(t-\tau))}\right)\\
&- \frac{bT_2V_2}{1+\alpha V_2} \left(\frac{1+\alpha V}{1+\alpha V_2}-1-\ln\frac{1+\alpha V}{1+\alpha V_2}\right)-\alpha bT_2 \frac{(V-V_2)^2}{(1+\alpha V)(1+\alpha V_2)^2}.\end{align*}

\section{Hopf Bifurcation Analysis}
For the bifurcation analysis we use the delay $\tau$ as a bifurcation parameter to find an interval in which the infected equilibria is stable and unstable out of the same margins.
Now to establish the Hopf bifurcation at $\tau =\tau _0$ we need to show that $\dfrac{d{\rm Re}\lambda (\tau _0)}{d\tau}>0$ differentiating \eqref{hiv8} with respect to $\tau$ we get
\[ \frac{d\lambda}{d\tau}=\frac{\lambda (b_1\lambda +b_0){\rm e}^{-\lambda \tau}}{3\lambda ^2+2a_2\lambda+a_1+b_1 {\rm e}^{-\lambda \tau}-(b_1\lambda +b_0){\rm e}^{-\lambda \tau}}\]
this gives
\[ \begin{aligned}
\left( \frac{d\lambda}{d\tau} \right) ^{-1}=&\frac{3\lambda ^2+2a_2\lambda+a_1+b_1{\rm e}^{-\lambda \tau}-(b_1\lambda +b_0){\rm e}^{-\lambda \tau}}{\lambda (b_1\lambda +b_0){\rm e}^{-\lambda \tau}}\\
=&\frac{3\lambda ^2+2a_2\lambda+a_1+b_1 {\rm e}^{-\lambda \tau}}{\lambda (b_1\lambda +b_0){\rm e}^{-\lambda \tau}}-\frac{\tau}{\lambda}\\
=&\frac{3\lambda ^3+2a_2\lambda ^2+a_1 \lambda+b_1\lambda {\rm e}^{-\lambda \tau}}{\lambda ^2 (b_1\lambda +b_0){\rm e}^{-\lambda \tau}}-\frac{\tau}{\lambda}\\
\end{aligned}\]
\[ \begin{aligned}
\left( \frac{d\lambda}{d\tau}\right)^{-1}=& \frac{2\lambda ^3+a_2\lambda ^2-a_0-b_0{\rm e}^{-\lambda \tau}}{\lambda ^2(b_1\lambda +b_0){\rm e}^{-\lambda \tau}}-\frac{\tau}{\lambda}\\
=& -\frac{2\lambda ^3+a_2\lambda ^2-a_0}{\lambda ^2(\lambda ^3+a_2\lambda ^2+a_1\lambda +a_0)}-\frac{b_0}{\lambda ^2(b_1\lambda+b_0)}-\frac{\tau}{\lambda}
\end{aligned}\]
It is important to be aware that we used \eqref{hiv8} in several equalities. Thus,
\[ \begin{aligned}
 {\rm sign} \left\{ \frac{d{\rm Re}\lambda}{d\tau}\right\}_{\lambda =i\omega _0}=&{\rm sign} \left\{  {\rm Re}\left( \frac{ d\lambda}{d\tau}\right)\right \}_{\lambda =i\omega _0} \\
=& {\rm sign} \bigg\{ {\rm Re} \left[ -\frac{2\lambda ^3+a_2\lambda ^2-a_0}{\lambda ^2(\lambda ^3+a_2\lambda ^2+a_1\lambda +a_0)}-\frac{b_0}{\lambda ^2(b_1\lambda+b_0)}-\frac{\tau}{\lambda} \right]_{\lambda =i\omega _0}\bigg\}\\
=& {\rm sign} \bigg\{ {\rm Re} \left[ -\frac{-2\omega _0 ^3i-a_2\omega ^2 i-a_0}{-\omega _0^2(-\omega _0^3i-a_2\omega _0^2 +a_1\omega _0i +a_0)}-\frac{b_0}{-\omega _0 ^2 (b_1\omega _0i +b_0)}-\frac{\tau _0}{\omega _0i} \right]\bigg\}\\
=& {\rm sign}\bigg\{ \frac{2\omega _0^6+(a_2^2-2a_1)\omega _0^4-a_0^2}{\omega _0^2[(a_2\omega _0^2-a_0)^2+(\omega _0^3-a_1\omega _0)^2]} +\frac{b_0^2}{\omega _0^2[b_0^2+(b_1\omega _0)^2]}\bigg\}\\
=& {\rm sign} \bigg\{ \frac{2\omega _0^6+(a_2^2-2a_1)\omega _0^4+b_0^2-a_0^2}{\omega _0^2[(a_2\omega _0^2-a_0)^2+(\omega _0^3-a_1\omega _0)^2]} \bigg\}\\
=& {\rm sign} \bigg\{ \frac{2\omega _0^6+(a_2^2-2a_1)\omega _0^4+(\omega _6+(a_2^2-2a_1)\omega _0^4+(a_1^2-2a_0a_2-b_1^2)\omega _0^2)}{\omega _0^2[(a_2\omega _0^2-a_0)^2+(\omega _0^3-a_1\omega _0)^2]} \bigg\}\\
=& {\rm sign} \bigg\{ \frac{3\omega _0^4+2(a_2^2-2a_1)\omega _0^2+a_1^2-2a_0a_2-b_1^2}{(a_2\omega _0^2-a_0)^2+(\omega _0^3-a_1\omega _0)^2} \bigg\}.
\end{aligned}\]
It is also important to notice that we used \eqref{hiv8}-\eqref{hiv11} in several equalities. Now to conclude that $\dfrac{d{\rm Re}\lambda }{d\tau}>0$ consider the next lemma
\begin{lemma}[\cite{wei}] Supposed that $x_1$, $x_2$, $x_3$ are the roots of equation $g(x)=x ^3+\alpha x^2+\beta x+\gamma =0$ ($\beta <0$), and $x_3$ is the largest positive simple root, then
\[ \frac{dg(x)}{dx}\bigg| _{x=x_3}>0.\]
\end{lemma}
In our case, considering $F(z)=z^3+Az^2+Bz+C=0$, defined in \eqref{hiv12}, and assuming $B<0$ and $\omega _0^2$ as the largest positive root we have
\[
\frac{d{\rm Re}\lambda}{d\tau}=\frac{\frac{dF(z)}{dz}}{(a_2\omega _0^2-a_0)^2+(\omega _0^3-a_1\omega _0)^2}>0.
\]
The above analysis can be summarized into the following theorem:
\begin{theorem}
Suppose that
\begin{itemize}
\item[(i)] $R_0>1$.
\end{itemize}
If either
\begin{itemize}
\item[(ii)] $C<0$
\end{itemize}
or
\begin{itemize}
\item[(iii)] $C\geq 0$ and $B <0$
\end{itemize}
is satisfied, and $\omega _0$ is the largest positive simple root of \eqref{hiv12} then the infected equilibrium $E_2$  of model \eqref{hiv3} is locally asymptotically stable when $\tau <\tau _0$ and unstable when $\tau >\tau _0$ where
\[
\tau _0=\frac{1}{\omega _0}\arccos \left[ \frac{b_0(a_2\omega_0 ^2-a_0)+b_1\omega_0 (\omega_0 ^3-a_1\omega_0 )}{b_0^2+b_1^2\omega_0^ 2}\right]
\]
when $\tau =\tau _0$, a Hopf bifurcation occurs; that is a family of periodic solutions bifurcates from $E_2$ as $\tau$ passes through the critical value $\tau _0$.
\end{theorem}
\section{Permanence}
\begin{lemma}
For any solution $(T(t),I(t),V(t))$ of system \eqref{hiv3}, we have
\[
\limsup _{t\rightarrow \infty} T(t)\leq T_0=\frac{T_{\max}}{2a}\left[ a-d+\sqrt{(a-d)^2+\frac{4as}{T_{\max}}} \right].
\]
Then there is a $t_1>0$ such that for any sufficiently small $\epsilon >0$, we have $T(t)\leq T_0+\epsilon$ for $t>t_1$.
\end{lemma}
The previous lemma follows, noting that for the first equation of \eqref{hiv3}, we have
\[ \dot T(t)\leq s-(d-a)T(t)-\frac{a}{T_{\max}}T^2(t)\]
\begin{theorem} \label{theorem}
There exist $M_I,\, M_V>0$ such that for any positive solution $(T(t), I(t), V(t))$ of system \eqref{hiv3},
\[I(t)<M_I, \; V(t)<M_V\]
for all large $t$.
\end{theorem}
\begin{proof}
Let $W(t)=T(t-\tau )+I(t)$, then
\[ \begin{aligned}
\dot W(t)=&\dot T(t-\tau)+\dot I(t)\\
=&s-dT(t-\tau )+aT(t-\tau )\left( 1-\frac{T( t-\tau )+I(t-\tau )}{T_{\max}}\right) -\frac{bT(  t-\tau )V(t-\tau)}{1+\alpha V(t-\tau)}+\\
& \frac{bT(t-\tau )V(t-\tau )}{1+\alpha V(t-\tau )}-\mu I+aI \left(  1-\frac{T+I}{T_{\max}}\right) \\
=& s-dT(t-\tau )+aT(t-\tau )\left(  1-\frac{T(  t-\tau )+I(t-\tau )}{T_{\max}}\right) -\mu I+aI\left( 1-\frac{T+I}{T_{\max}}\right) \\
=& s-dT (t-\tau )  +aT\left(
t-\tau\right)  -\frac{aT^{2}\left(  t-\tau\right)  }{T_{\max}}-\frac{aT\left(
t-\tau\right)  I\left(  t-\tau\right)  }{T_{\max}}-\mu I+aI-\frac{aTI}%
{T_{\max}}\\
&-\frac{aI^{2}}{T_{\max}}.
\end{aligned}\]
Using that, $-\frac{a}{T_{\max}}\left(  T\left(  t-\tau\right)-\frac{T_{\max}}{2}\right) ^{2} +\frac{aT_{\max}}{4}=-\frac{aT^{2}\left(  t-\tau\right)
}{T_{\max}}+aT\left(  t-\tau\right)  $ and $-\frac{a}{T_{\max}}\left(I-\frac{T_{\max}}{2}\right)  ^{2}+\frac{aT_{\max}}{4}=-\frac{aI^{2}}{T_{\max}
}+aI$, we get
\begin{align*}
\dot{W}\left(  t\right)   &  =s-dT\left(  t-\tau\right)  -\frac{a}{T_{\max}%
}\left(  T\left(  t-\tau\right)  -\frac{T_{\max}}{2}\right)  ^{2}-\frac
{a}{T_{\max}}\left(  I-\frac{T_{\max}}{2}\right)  ^{2}+\frac{aT_{\max}}{2}+\\
&  -\frac{aT\left(  t-\tau\right)  I\left(  t-\tau\right)  }{T_{\max}}-\mu
I-\frac{aTI}{T_{\max}},
\end{align*}
then
\begin{align*}
\dot{W}\left(  t\right)   &  =-dT\left(  t-\tau\right)  -\mu I-\frac
{a}{T_{\max}}\left(  T\left(  t-\tau\right)  -\frac{T_{\max}}{2}\right)
^{2}-\frac{a}{T_{\max}}\left(  I-\frac{T_{\max}}{2}\right)  ^{2}%
+\frac{aT_{\max}+2s}{2}+\\
&  -\frac{aT\left(  t-\tau\right)  I\left(  t-\tau\right)  }{T_{\max}}
-\frac{aTI}{T_{\max}}\\
&\leq -dT(t-\tau)-\mu I(t)+\frac{aT_{\max}+2s}{2},
\end{align*}
therefore
\[
\dot{W}\left(t\right)\leq -d W\left(t\right) +\frac {aT_{\max}+2s}{2}
\]
where $d \leq \mu$. Hence, we get the boundness of $W (t)$
\[ \limsup _{t\rightarrow \infty}W(t)=\frac{aT_{\max}/2+s}{h},\]
that is, there exist $t_2>0$ and $M_1>0$ such that $W(t)<M_1$ for $t>t_2$. Then $I(t)$ has an ultimately upper bound $M_I$.\\
It follows from the third equation of system \eqref{hiv3} that $V(t)$ has an ultimately upper bound, say $M_V$. Then the assertion of theorem follows and the proof is complete. \end{proof}
Define
\[ \Omega =\{ (T,I,V):0\leq T\leq T_0, \; 0\leq I\leq M_I,\; 0\leq V\leq M_V\} \]
System \eqref{hiv3} satisfies, for some $t_1>0$,
\[
\dot T\geq s-dT+aT\left( 1-\frac{T+M_I}{T_{\max}}\right)-\frac{b T}{\alpha}
\]
which implies that
\[
\liminf _{t\rightarrow \infty} T(t)\geq \frac{T_{\max}}{2a}\left[ a-d-\frac{b}{\alpha} -\frac{aM_I}{T_{\max}}+\sqrt{\left( a-d-\frac{b}{\alpha}-\frac{aM_I}{T_{\max}}\right)^2+\frac{4as}{T_{\max}}}\right]
\]
Now we shall prove that the instability of $E_1$ implies that system \eqref{hiv3} is permanent.
\begin{definition}
System \eqref{hiv3} is said to be uniformly persistent, if there is an $\eta >0$ (independent of the initial data) such that every solution $(T(t),I(t),V(t))$ with initial condition of system \eqref{hiv3} satisfies $\liminf _{t\rightarrow \infty} T(t)\geq \eta$, $\liminf _{t\rightarrow \infty} I(t)\geq \eta$, $\liminf _{t\rightarrow \infty} V(t)\geq \eta$.
\end{definition}
For dissipative system uniform persistence is equivalent to the permanence.
\begin{theorem} \label{permanence}
System \eqref{hiv3} is permanent provided $R_0>1$.
\end{theorem}
We present the persistence theory for infinite dimensional system from Hale \cite{hale}. Let $X$ be a complete space metric. Suppose that $X^0 \subset X$, $X_0 \subset X$, $X^0\cap X_0=\emptyset$, $X=X^0\cup X_0$. Assume that $Y(t)$ is $C^0$-semigroup on $X$ satisfying
\begin{equation} \label{hiv9}
\bigg\{ \begin{matrix}
Y(t):X^0\rightarrow X^0,\\
Y(t):X_0\rightarrow X_0.
\end{matrix}
\end{equation}
Let $Y_b(t)=Y(t)|_{X_0}$ and let $A_b$ be the global attractor for $Y_b(t)$.
\begin{lemma} \label{lemma}
Suppose that $Y(t)$ satisfies \eqref{hiv9} and we have the following:
\begin{enumerate}
\item there is a $t_0\geq 0$ such that $Y(t)$ is compact for $t>t_0$,
\item $Y(t)$ is a point dissipative in $X$,
\item $\bar A_b=\cup _{x\in A_b}\omega (x)$ is isolated and has an acyclic covering $\overline{M}$, where $\overline{M}=\{ M_1,M_2,\ldots ,M_n\}$,
\item $W^s(M_i)\cap X^0=\emptyset$, for $i=1,2,\ldots ,n$.
\end{enumerate}
Then $X_0$ is a uniform repellor with respect to $X^0$, i.e., there is an $\epsilon >0$ such that for any $x\in X^0$
\[ \liminf _{t\rightarrow \infty} d(Y(t)x,X_0)\geq \epsilon\], where $d$ is the distance of $Y(t)x$ from $X_0$.
\end{lemma}
We now prove theorem \ref{permanence}.
\begin{proof}[Proof of Theorem \ref{permanence}]
We begin by showing that the boundary planes of $\mathbb{R} ^3_+$ repel the positive solutions of system \eqref{hiv3} uniformly. Let us define
\[ C_0=\{ (\psi , \phi _1, \phi_2) \in C([-\tau ,0],\mathbb{R}^3_+): \psi (\theta )\not = 0, \phi _1(\theta)=\phi _2(\theta)=0, \;(\theta \in [-\tau ,0]) \}.\]
If $C^0=int C([-\tau ,0],\mathbb{R}^3_+)$, it suffices to show that there exist an $\epsilon _0$ such that any solution $u_t$ of system \eqref{hiv3} initiating from $C^0$, $\liminf _{t\rightarrow +\infty}d(u_t,C_0)\geq \epsilon _0$. To this end, we verify below that the conditions of lemma \ref{lemma} are satisfied. It is easy to see that $C^0$ and $C_0$ are positively invariant. Moreover, conditions (1) and (2) of lemma \ref{lemma} are satisfied. Thus, we only need to verify the conditions (3) and (4). There is a constant solution $E_1$ in $C_0$, to $T(t)=T_0$, $I(t)=V(t)=0$. If $(T(t),I(t),V(t))$ is a solution of system \eqref{hiv3} initiating from $C_0$, then $T(t)\rightarrow T_0$, $I(t)\rightarrow 0$, $V(t)\rightarrow 0$, as $t\rightarrow +\infty$. It is obvious that $E_1$ is an isolated invariant. Now, we show that $W^s(E_1)\cap C^0=\emptyset $. Assuming the contrary, then there exist a positive solution $(\tilde{T}(t), \tilde{I}(t), \tilde{V}(t))$ of system \eqref{hiv3} such that $((\tilde{T}(t), \tilde{I}(t), \tilde{V}(t))) \rightarrow (T_0,0,0)$ as $t\rightarrow \infty$. Let us choose $\epsilon >0$ small enough and $t_0>0$ sufficiently large such that
\[ T_0-\epsilon <\tilde{T} (t)<T_0+\epsilon, \;  0<\tilde I(t)<\epsilon \]
for $t>t_0-\tau$. Then we have for $t>t_0$
\[
\Bigg\{ \begin{array}{l}
\dot{ \tilde{I}}(t)\geq b(T_0-\epsilon)\tilde V(t-\tau)+\left( -\mu +a\left( 1-\dfrac{T_0+\epsilon +\epsilon}{T_{\max}}\right)\right) \tilde I(t),\\
\dot{ \tilde{V}}(t)=p\tilde{I}(t)-c\tilde{V}(t)
\end{array}
\]
Let us consider the matrix defined by
\[ A_\epsilon =\left( \begin{matrix}
 -\mu +a\left( 1-\dfrac{T_0+\epsilon +\epsilon}{T_{\max}}\right) &  b(T_0-\epsilon )\\ p & -c
\end{matrix} \right).\]

Since $A_\epsilon$ admits positive off-diagonal elements, Perron-Frobenius theorem implies that there is a positive eigenvector $\hat {V}$ for the maximum eigenvalue $\lambda _1$ of $A_\epsilon$. Moreover, since $R_0>1$, then $c\mu -ac\left( 1-\dfrac{T_0+2\epsilon}{T_{\max}} \right)- bp(T_0-\epsilon)<0$ for $\epsilon$ small enough, by a simple computation we see that $\lambda _1$ is positive.\\
Let us consider
\begin{equation} \label{hiv13}
\bigg\{ \begin{array}{l}
\dot I(t)=b(T_0-\epsilon)V(t-\tau ) +\left( -\mu +a\left( 1-\dfrac{T_0+\epsilon +\epsilon}{T_{\max}}\right) \right)I(t)\\
\dot V(t)=pI(t)-cV(t)
\end{array}.
\end{equation}
Let $v=(v_1,v_2)$ and $l>0$ be small enough such that
\[ \begin{aligned}
lv_1&<\tilde{I}(t_0+\theta),\\
lv_2&<\tilde{V}(t_0+\theta),
\end{aligned}\]
for $\theta \in [-\tau ,0]$ if $(I(t),V(t))$ is a solution of system \eqref{hiv13} satisfying $I(t)=lv_1$, $V(t)=lv_2$ for $t_0-\tau \leq t\leq t_0$.

Since the semiflow of system \eqref{hiv13} is monotone and $A_\epsilon v>0$, it follows that $I(t)$ and $V(t)$ are strictly increasing and $I(t)\rightarrow \infty$, $V(t)\rightarrow \infty$ as $t\rightarrow \infty$. Note that $\tilde{I}\geq I(t)$, $\tilde{V}(t)\geq V(t)$ for $t>t_0$. We have $\tilde{I}(t)\rightarrow \infty$, $\tilde{V}(t)\rightarrow \infty$ as $t\rightarrow \infty$. At this time, we are able to conclude form theorem \ref{theorem} that $C_0$ repels the positive solutions of system \eqref{hiv3} uniformly. Incorporating this into lemma \ref{lemma} and theorem \ref{theorem}, we know that the system \eqref{hiv3} is permanent.
\end{proof}

\section{Estimation of the length of delay to preserve stability}

Let $T(t)=T_2+X(t),I(t)=I_2+Y(t),V(t)=V_2+Z(t)$.\\
We consider the linearized system \eqref{hiv3} about the equilibrium $E_2$ and get
\begin{equation} \begin{aligned} \label{hiv14}
\dfrac{dX}{dt}=&\left( a-d-\frac{2aT_2}{T_{\max}}-\frac{aI_2}{T_{\max}}-\frac{bV_2}{T_{\max}}\right)X -\dfrac{aT_2}{T_{\max}}Y-\frac{bT_2}{(1+\alpha V_2)^2}Z \\
\dfrac{dY}{dt}=& -\frac{aI_2}{T_{\max}}X+ \left( a-\mu -\frac{aT_2}{T_{\max}}-\frac{2aI_2}{T_{\max}} \right)Y+\frac{bV_2}{1+\alpha V_2}X_{\tau} +\frac{bT_2}{(1+\alpha V_2)^2}Z_\tau \\
\frac{dZ}{dt}=&pY-cZ
\end{aligned} \end{equation}

Taking Laplace transform of the system given by \eqref{hiv14}, we get
\begin{equation} \begin{aligned} \label{hiv15}
s \mathcal{L}[X]-X(0)=&\left( a-d-\frac{2aT_2}{T_{\max}}-\frac{aI_2}{T_{\max}}-\frac{bV_2}{1+\alpha V_2} \right)L[X]-\frac{aT_2}{T_{\max}}
\mathcal{L}[Y]\\
&-\frac{bT_2}{(1+\alpha V_2)^2} \mathcal{L}[Z]\\
s\mathcal{L}[Y]-Y(0)=&-\frac{aI_2}{T_{\max}}
\mathcal{L}[X]- \left( a-\mu -\frac{aT_2}{T_{\max}}-\frac{2aI_2}{T_{\max}}\right)
\mathcal{L}[Y]+\frac{bV_2}{1+\alpha V_2}
\mathcal{L}[X_\tau]\\
&+\frac{bT_2}{(1+\alpha V_2)^{2}}\mathcal{L}[Z_\tau]\\
s\mathcal{L}[Z]-Z(0)=&p\mathcal{L}[Y]-c
\mathcal{L}[Z]
\end{aligned} \end{equation}
The expressions $\mathcal{L}[Z_{\tau}]$ and $\mathcal{L}[X_{\tau}]$ are equivalent to
\[\mathcal{L}\left[  X_{\tau}\right]  =\int_{0}^{\infty}e^{-st}X\left(  t-\tau\right)
dt=\int_{0}^{\tau}e^{-st}X\left(  t-\tau\right)  dt+\int_{\tau}^{\infty
}e^{-st}X\left(  t-\tau\right)dt \]
taking $t=t_{1}+\tau$ we can express the last equation as
\begin{align*}\mathcal{L}[ X_{\tau}]&=\int_{-\tau}^{0}e^{-s\left(  t_{1}+\tau\right)
}X\left(  t_{1}\right)  dt_{1}+\int_{0}^{\infty}e^{-s\left(  t_{1}
+\tau\right)  }X\left(  t_{1}\right)  dt_{1}\\
&  =e^{-s\tau}\int_{-\tau}^{0}e^{-st_{1}}X\left(  t_{1}\right)  dt_{1}
+e^{-s\tau}\int_{0}^{\infty}e^{-st_{1}}X\left(  t_{1}\right)  dt_{1}\\
&  =e^{-s\tau}K_{1}+e^{-s\tau}%
\mathcal{L}[X].
\end{align*}
In the same way
\[\mathcal{L}[Z_{\tau}]=\int_{0}^{\infty}e^{-st}Z\left(  t-\tau\right) dt=\int_{0}^{\tau}e^{-st}Z\left( t-\tau\right)  dt+\int_{\tau}^{\infty }e^{-st}Z\left(  t-\tau\right)  dt.
\]
We have
\begin{align*}\mathcal{L}
\left[  Z_{\tau}\right]   &  =\int_{-\tau}^{0}e^{-s\left(  t_{1}+\tau\right)
}Z\left(  t_{1}\right)  dt_{1}+\int_{0}^{\infty}e^{-s\left(  t_{1}
+\tau\right)  }Z\left(  t_{1}\right)  dt_{1}\\
&  =e^{-s\tau}\int_{-\tau}^{0}e^{-st_{1}}Z\left(  t_{1}\right)  dt_{1}
+e^{-s\tau}\int_{0}^{\infty}e^{-st_{1}}Z\left(  t_{1}\right)  dt_{1}\\
&  =e^{-s\tau}K_{2}+e^{-s\tau}
\mathcal{L}[Z].\end{align*}
Replacing  $\mathcal{L}[X_{\tau}]$ and $\mathcal{L}[Z_{\tau} ]$ on system \eqref{hiv15} we can clear $\mathcal{L}[X]$ and $\mathcal{L}[Z]$ and we can write
\[ (A-sI)\left( \begin{array}{c}
\mathcal{L}[X]\\ \mathcal{L}[Y]\\ \mathcal{L}[Z]
\end{array} \right)=B\]
where
\[ \begin{aligned}
A=&
\left( \begin{matrix}
a-d-\frac{2aT_2}{T_{\max}}-\frac{aI_2}{T_{\max}}-\frac{bV_2}{1+\alpha V_2} & -\frac{aT_2}{T_{\max}} & -\frac{bT_2}{(1+\alpha V_2)^2} \\
-\frac{aI_2}{T_{\max}}+\frac{bV_2}{1+\alpha V_2}e^{-s\tau} & a-\mu -\frac{aT_2}{T_{\max}}-\frac{2aI_2}{T_{\max}} & \frac{bV_2}{(1+\alpha V_2)^2}e^{-s\tau} \\
0 & p & c
\end{matrix} \right) \\
B=& \left( \begin{matrix}
X(0) \\ Y(0) +(K_1+K_2)e^{-s\tau} \\ Z(0)
\end{matrix} \right).
\end{aligned} \]
The inverse Laplace transformations of $\mathcal{L}[X]$, $\mathcal{L}[Y]$ and $\mathcal{L}[Z]$ will have terms which exponentially increase with time if $\mathcal{L}[X]$, $\mathcal{L}[Y]$ and $\mathcal{L}[Z]$ have poles with positive real parts. For $E_2$ to be locally asymptotically stable, a necessary and sufficient condition is that all poles of $\mathcal{L}[X]$, $\mathcal{L}[Y]$ and $\mathcal{L}[Z]$ have negative real parts. We will employ the Nyquist criteria, which states that if $X$ is the arc length of a curve encircling the right half plane, the curve $\mathcal{L}[X]$ will encircle the origin a number of times equal to the difference between the numbers of poles and zeroes of $\mathcal{L}[X]$ in the right half plane. This criteria is applied to $X$, $Y$ and $Z$.Let
\[ F\left(  s\right)  =s^{3}+a_2s^{2}+a_1s+a_0+(b_1s+b_0)e^{-s\tau} \]
obtained from the Laplace transform. Note that $F(s)=0$ is the characteristic equation of system \eqref{hiv3} on the equilibrium $E^*$ and the zeroes are the poles of $\mathcal{L}[X]$, $\mathcal{L}[Y]$ and $\mathcal{L}[Z]$. The conditions for local asymptotic stability of $E_2$ are given in \cite{nyquist}
\begin{equation} \begin{aligned} \label{hiv16}
\Re[F(iv_{0})]&=0 \\
\Im[F(iv_{0})]&>0,
\end{aligned}\end{equation}
and $v_{0}$ is the smallest positive root of the first equation of \eqref{hiv16}. In our case, \eqref{hiv16} gives
\begin{eqnarray}
-a_2v_0^2+a_0+b_1v_0 \sin (v_0\tau)+b_0\cos (v_0\tau)&=&0 \label{hiv17} \\
-v_0^3+a_1v_0+b_1v_0\cos (v_0\tau)-b_0\sin (v_0\tau)&>&0.  \label{hiv18}
\end{eqnarray}

If \eqref{hiv17} and \eqref{hiv18} are satisfied simultaneously, they are sufficient conditions to guarantee stability. We shall apply them to get an estimate on the length of delay. Our aim is to find an upper bound $v^+$ on $v_0$, independent of $\tau$ and then to estimate $\tau$ so that \eqref{hiv18} hold for all values of $v$, $0\leq v\leq v^+$ and hence in particular $v=v_0$. We rewrite \eqref{hiv17} as
\begin{equation} \label{hiv185}
a_2v_0^2=a_0+b_1v_0\sin (v_0\tau)+b_0\cos (v_0\tau ).
\end{equation}
Maximizing $a_0+b_1v_0\sin (v_0\tau)+b_0\cos (v_0\tau )$ subject to
\[|\sin (v_0\tau )|\leq 1, \quad |\cos (v_0\tau)|\leq 1,\]
we obtain
\begin{equation} \label{hiv19}
a_2v_0^2\leq |a_0|+|b_1|v_0+|b_0|.
\end{equation}
Hence, if
\[ v^+=\dfrac{|b_1|+\sqrt{b_1^2+4a_2(|a_0|+|b_0|)}}{2a_2}\]
then clearly from \eqref{hiv19} we have $v_0\leq v^+$.\\
We can rewrite \eqref{hiv18} as
\begin{equation} \label{hiv20}
v_0^2<a_1+b_1\cos (v_0\tau)-\dfrac{a_2b_0}{v_0}\sin (v_0\tau).
\end{equation}
Replacing \eqref{hiv185} in \eqref{hiv20} and rearranging we get
\[ (b_0-a_1b_1)(\cos (v_0\tau)-1)+\left( b_1v_0+\frac{a_2b_0}{v_0}\right) \sin (v_0\tau)<a_2a_1-a_0+a_1b_1-b_0.\]
Using the bounds
\[ \begin{aligned}
\left| b_1v_0+\frac{a_2b_0}{v_0}\right||\sin (v_0\tau)|& \leq \left| b_1v^++\frac{a_2b_0}{v^+}\right|(v^+\tau)=(|b_1|(v^+)^2+|a_2b_0|)\tau, \\
\left| b_0-a_1b_1\right||\cos (v_0\tau)-1| & \leq 2\left| b_0-a_1b_1\right| \sin^2 \left(\frac{v_0\tau}{2}\right) \leq \frac{\left| b_0-a_1b_1\right|}{2}(v^+)^2\tau ^2,
\end{aligned} \]
we obtain from \eqref{hiv20} $K_1\tau ^2+K_2\tau <K_3$, where
\[K_1=\frac{|b_2-a_1b_1|}{2}(v^+)^2, \quad K_2=|b_1|(v^+)^2+|a_2b_2|, \quad K_3=a_2a_1+a_1b_1-a_0-b_2\]
thus if $K_1\tau 2+K_2\tau<K_3$ holds, then the inequality \eqref{hiv17} is satisfied. A positive root of $K_1\tau ^2+K_2\tau=K_3$ is given by
\[ \tau _+=\frac{1}{2K_1}(-K_2+\sqrt{K_2^2+4K_1K_3}).\]
For $0\leq \tau \leq \tau _+$, the Nyquist criteria holds. $\tau _+$ gives estimate for the length of delay for which stability is preserved. Here $\tau _+$ is dependent of the system parameters. Hence we can conclude that the estimate for the delay is totally dependent on system parameters for which the equilibrium $E_2$ is locally asymptotically stable.
\begin{theorem}
If there exist a parameter $0\leq \tau \leq \tau _+$ such that $K_1\tau ^2+K_2\tau <K_3$, then $\tau _+$ is the maximum value (length of delay) of $\tau$ for which $E_2$ is asymptotically stable.
\end{theorem}

\section{Numerical Simulations}
To explore the behaviour of the system \eqref{hiv3} and illustrate the stability of equilibria solutions we used, dde23 \cite{dde23}, based on Runge-Kutta methods. We consider the values for the parameters as in \cite{dahari1}.

In figure 1 we illustrate the stability of $E_1$  with the following parameters $s=8\times 10^5$, $d=4.7\times 10^{-3}$, $\mu=0.35$, $a=1$, $T_{\max}=0.7\times 10^{7}$, $b=0.6\times 10^{-7}$, $c=5.9$, $p=5.4$ $\alpha=0.001$, with this values for the parameters, $R_0=0.9237$ so we are under conditions of theorem 1 and 3. We show the dynamic of solutions for several values of $\tau$ and we can appreciate that the solution approximates to equilibrium with oscillations as the values of $\tau$ increases.

\begin{figure}[H]
\centering
\subfigure[Healthy, $T(t)$ ]{\includegraphics[scale=0.3]{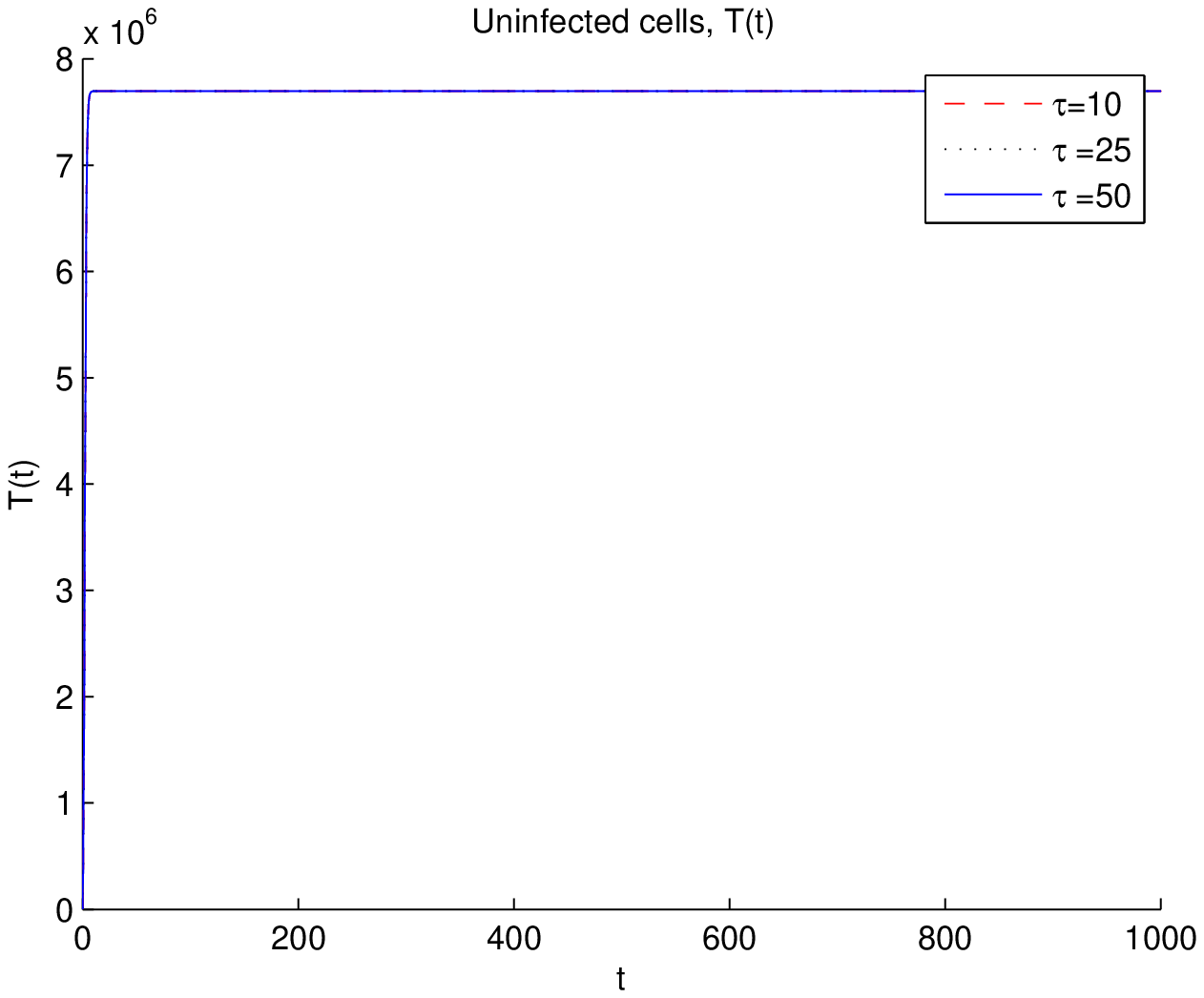}}
\subfigure[Infected, $I(t)$]{\includegraphics[scale=0.3]{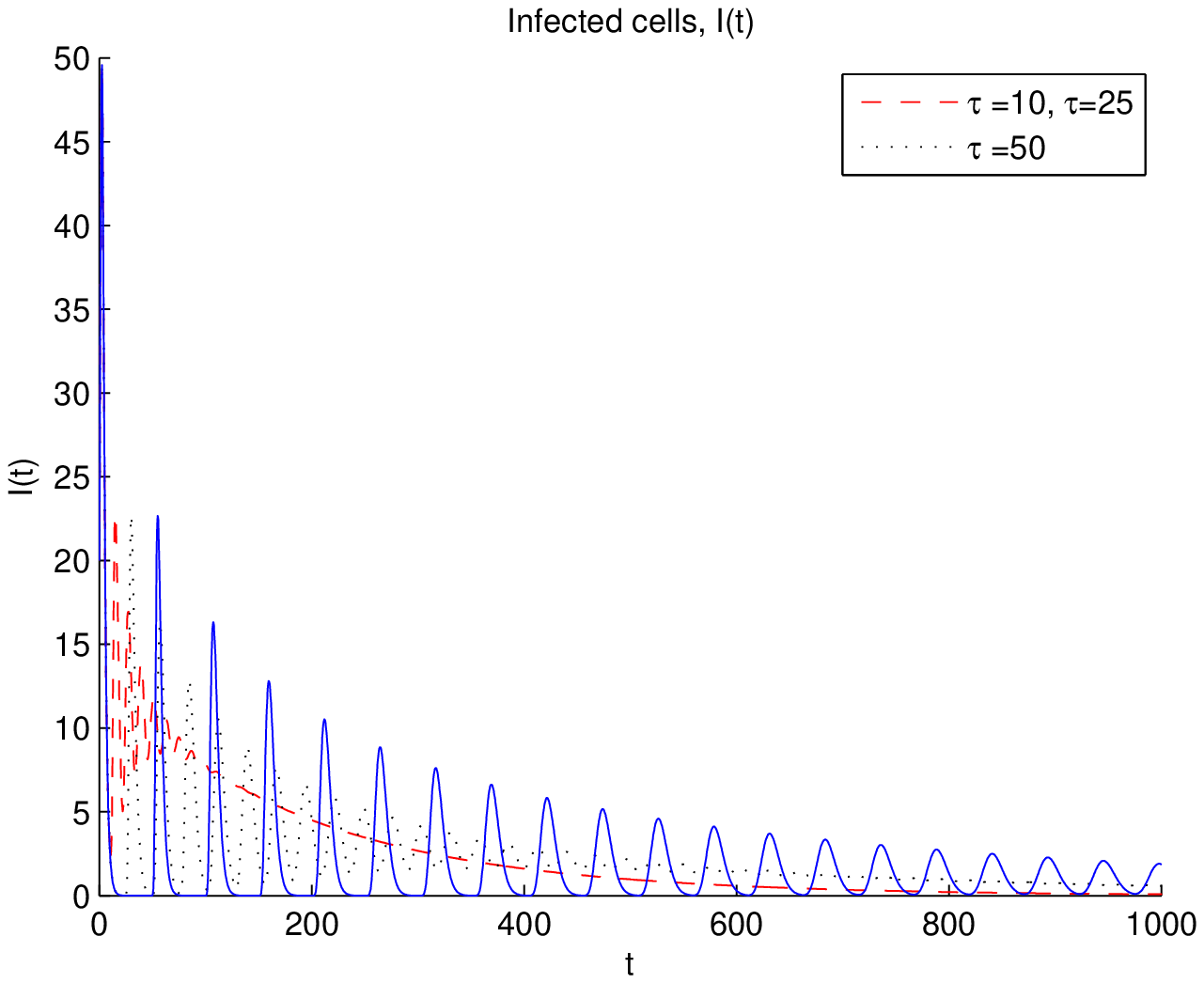}}
\subfigure[Virus, $V(t)$]
{\includegraphics[scale=0.3]{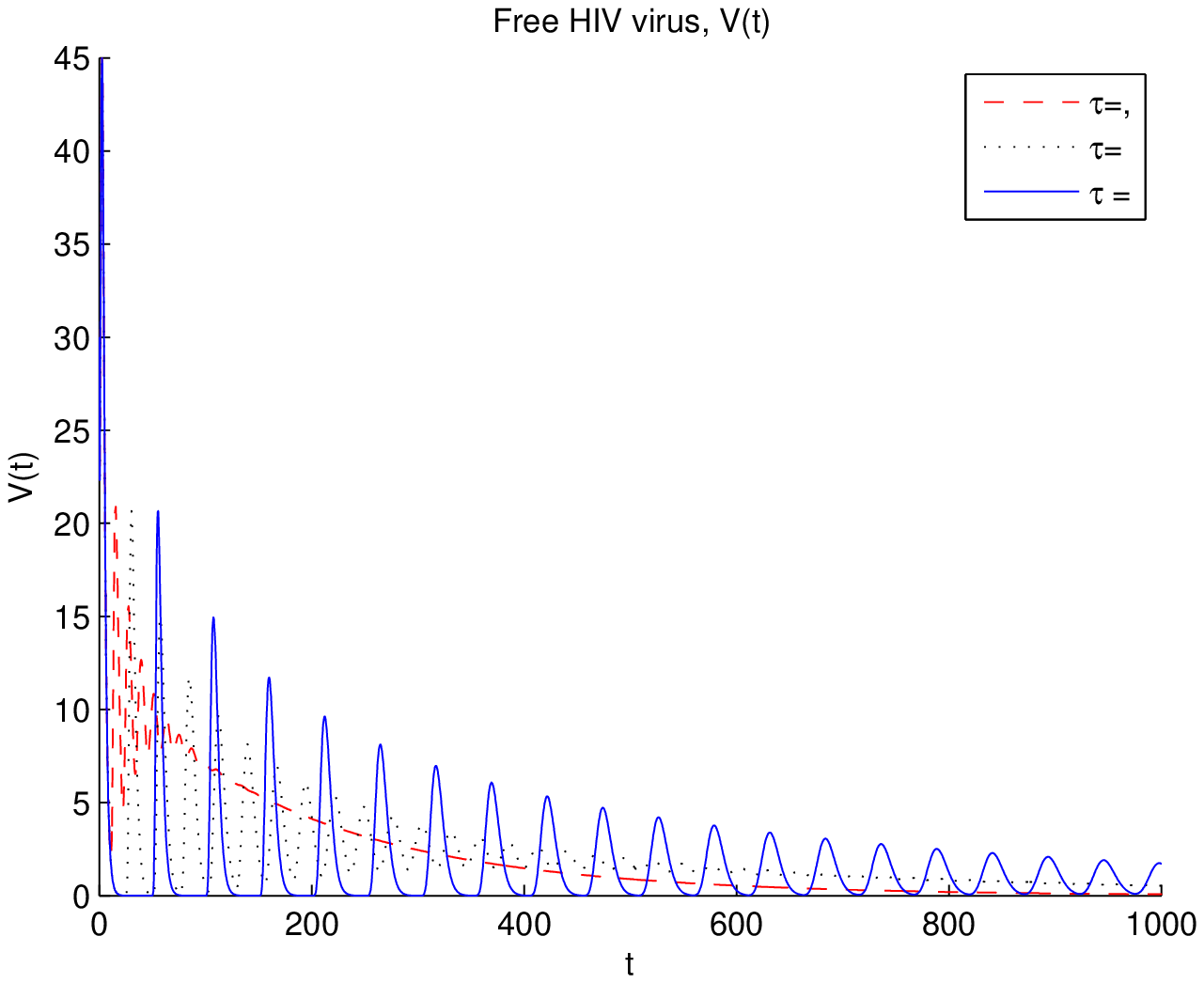}}
\caption{Stability of the $E_1$ equilibrium with several values of $\tau$}
\label{figure1}
\end{figure}

In figure 2 we see the same dynamic for system \eqref{hiv3} with a large time and different delay $\tau$, so we can conclude that delay has no effect on stability of infection-free equilibrium for our model.

\begin{figure}[H]
\centering
\subfigure[Dynamics, $T(t)$, $I(t)$, $V(t)$]{\includegraphics[scale=0.3]{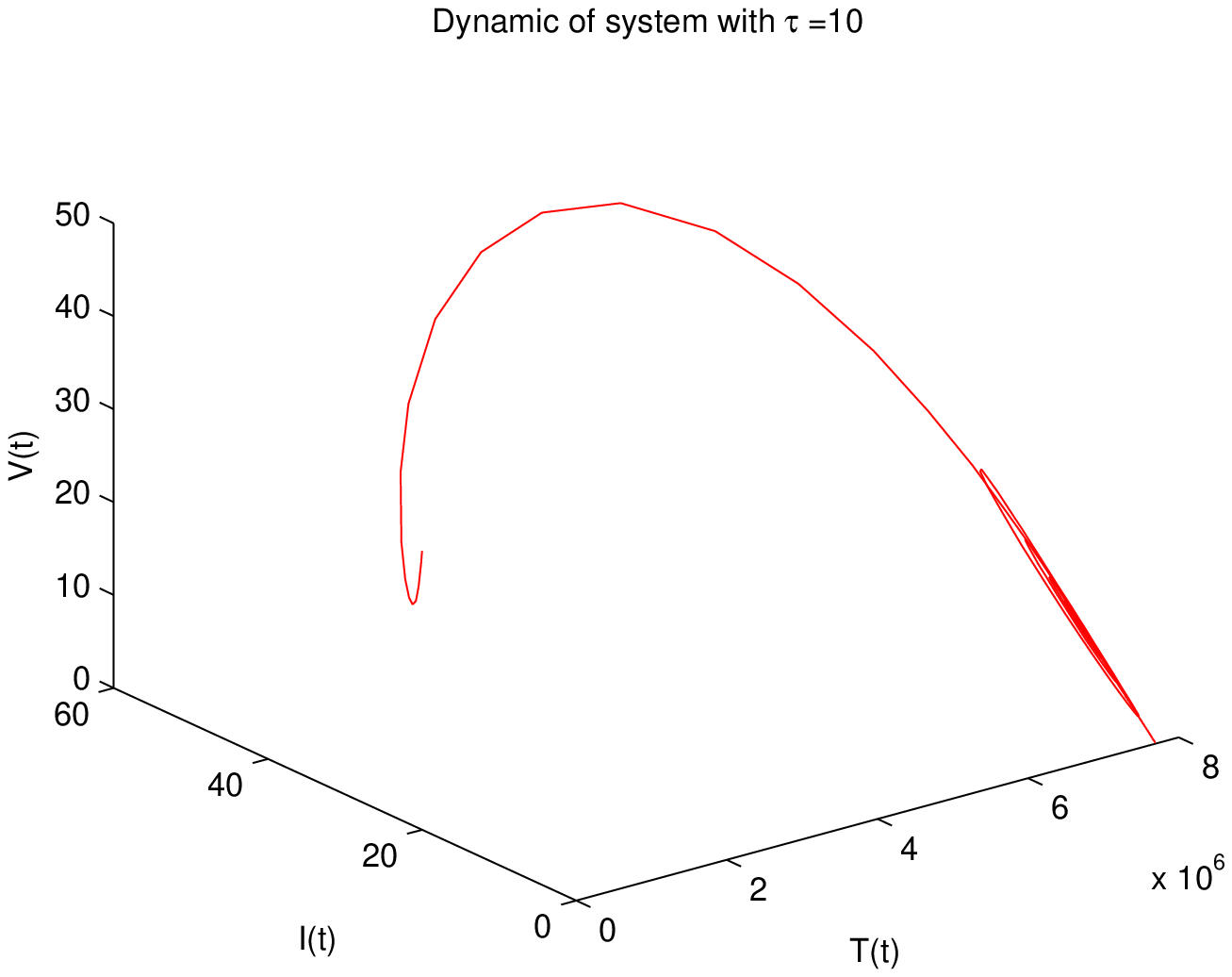}}
\subfigure[Dynamics, $T(t)$, $I(t)$, $V(t)$]{\includegraphics[scale=0.3]{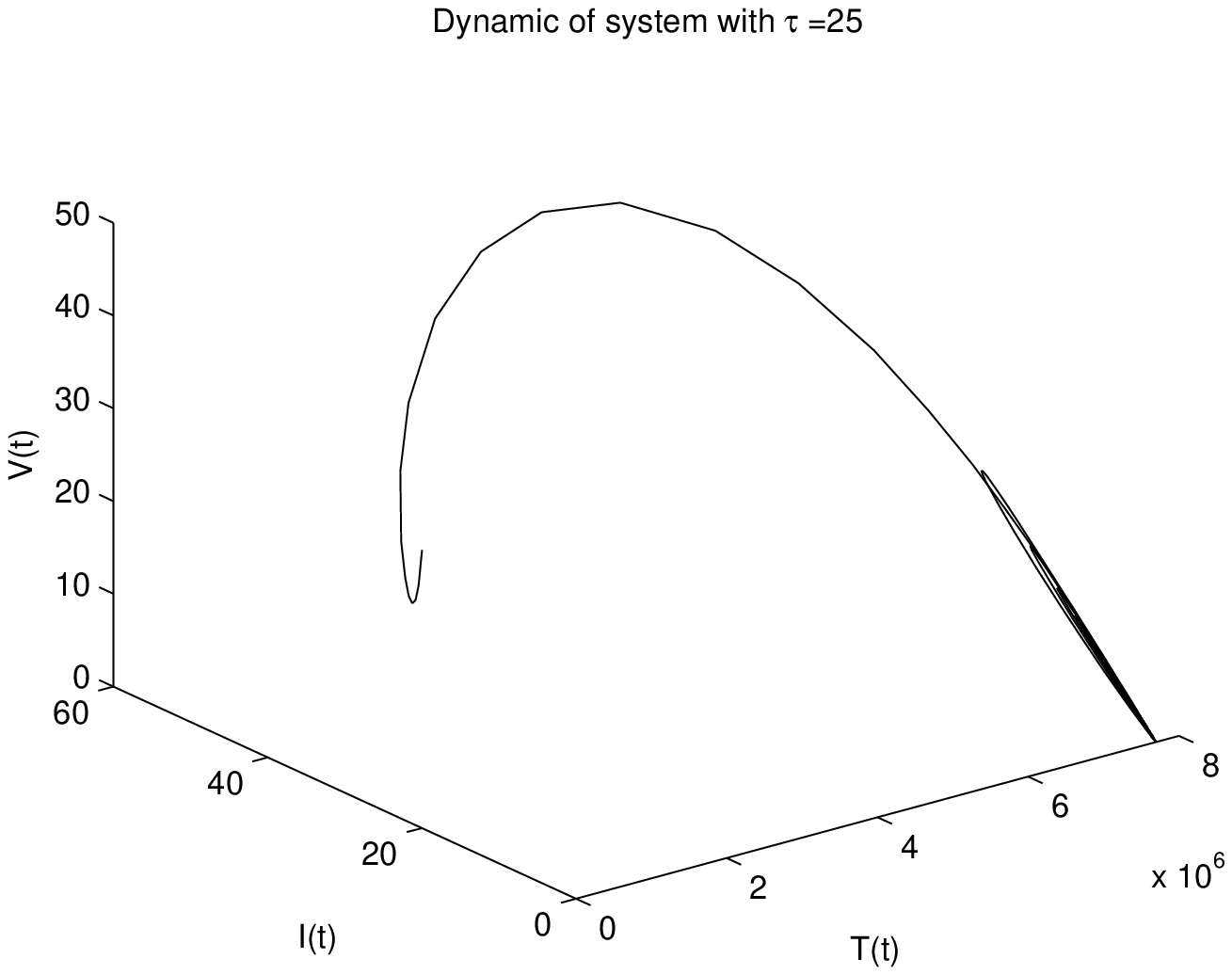}}
\subfigure[Dynamics, $T(t)$, $I(t)$, $V(t)$]{\includegraphics[scale=0.3]{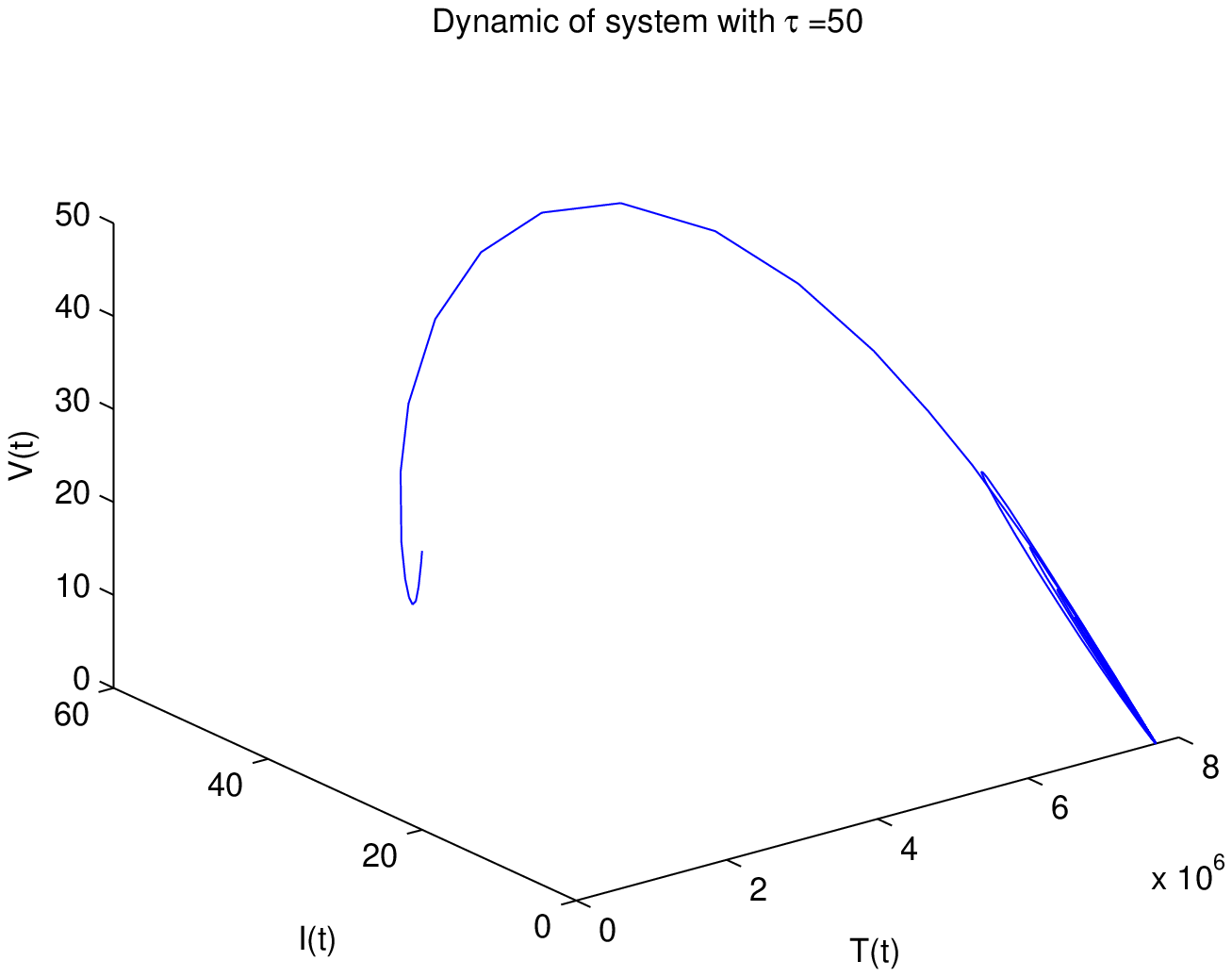}}
\caption{Phase space for different values of $\tau$, which illustrates the stability of infection-free equilibrium $E_1$}
\label{figure2}
\end{figure}

For figure 3 we consider the following values for parameters $s=8\times 10^5$, $d=4.7\times 10^{-3}$,$\mu=0.3$, $a=2$, $T_{\max}=0.7\times 10^7$, $b=0.6\times 10^{-7}$, $c=5.9$, $p=5.4$ and $\alpha=0.001$, in this case the equilibrium is $(7.363787665\times 10^6, 1.202557883, 1.100646198)$ with $R_0=1.0014820$, note that this case satisfies the conditions for global stability of $E_2$, as is establish in theorem 4.

\begin{figure}[H]
\centering
\subfigure[Healthy, $T(t)$ ]{\includegraphics[scale=0.3]{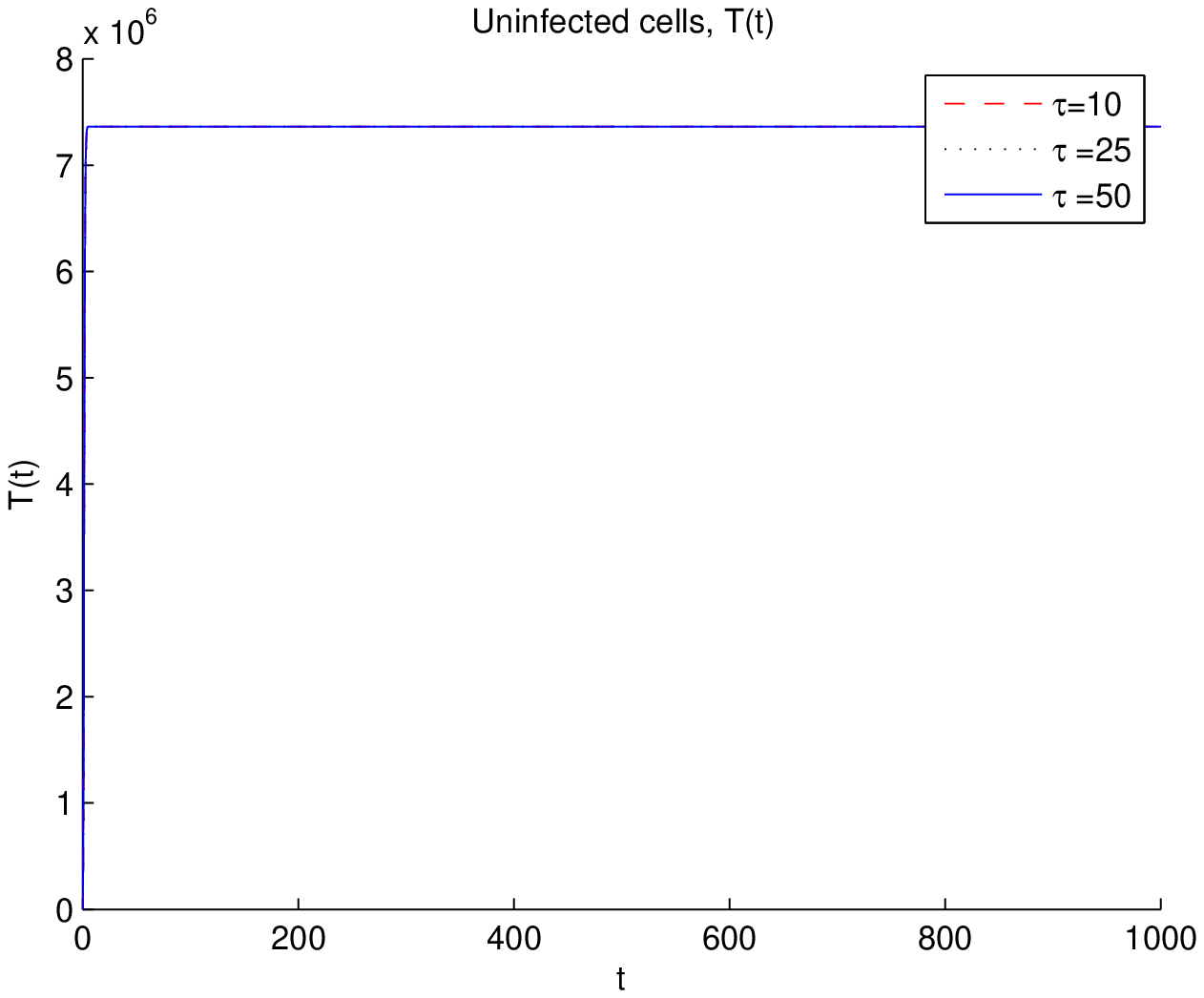}}
\subfigure[Infected, $I(t)$]{\includegraphics[scale=0.3]{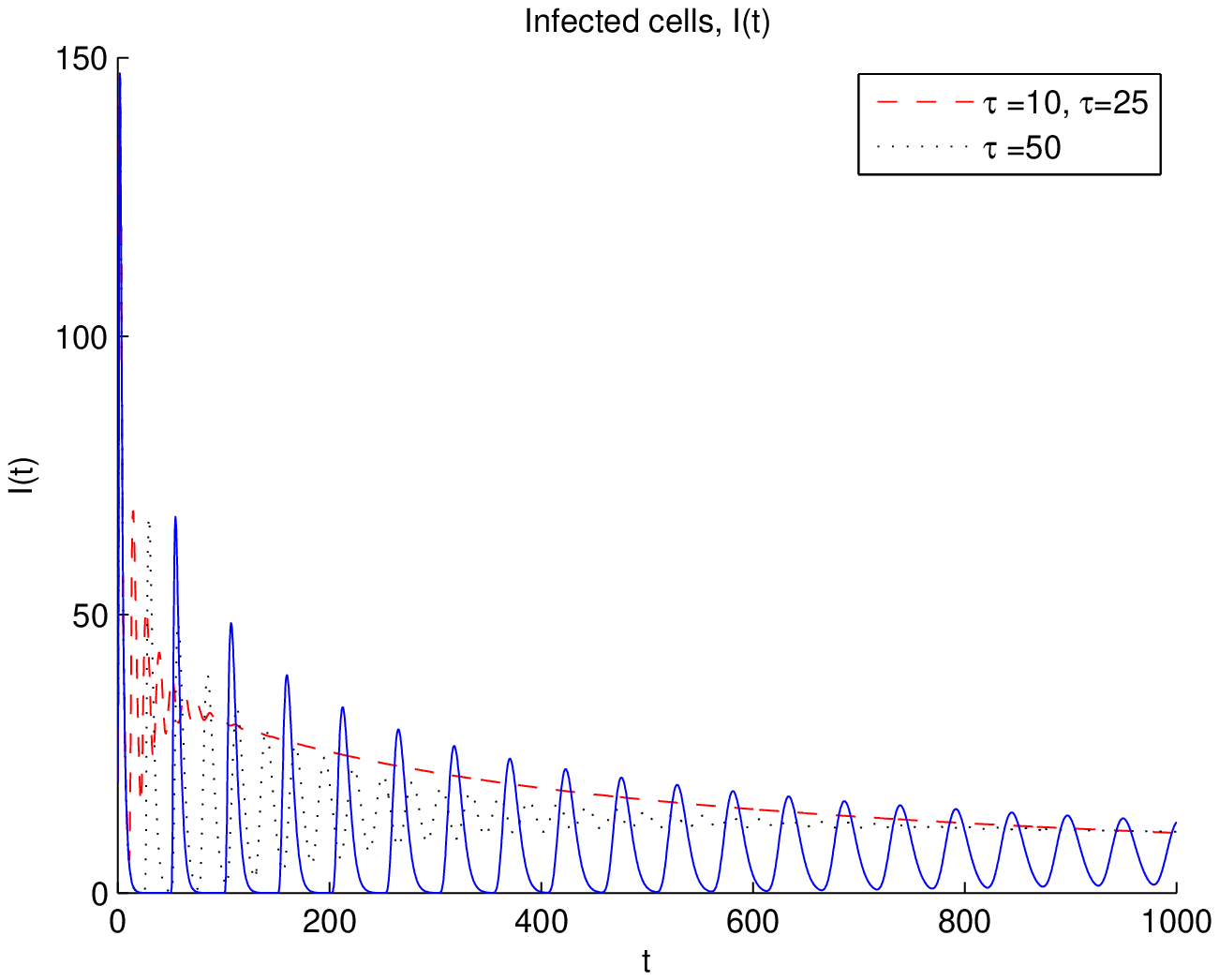}}
\subfigure[Virus, $V(t)$]
{\includegraphics[scale=0.3]{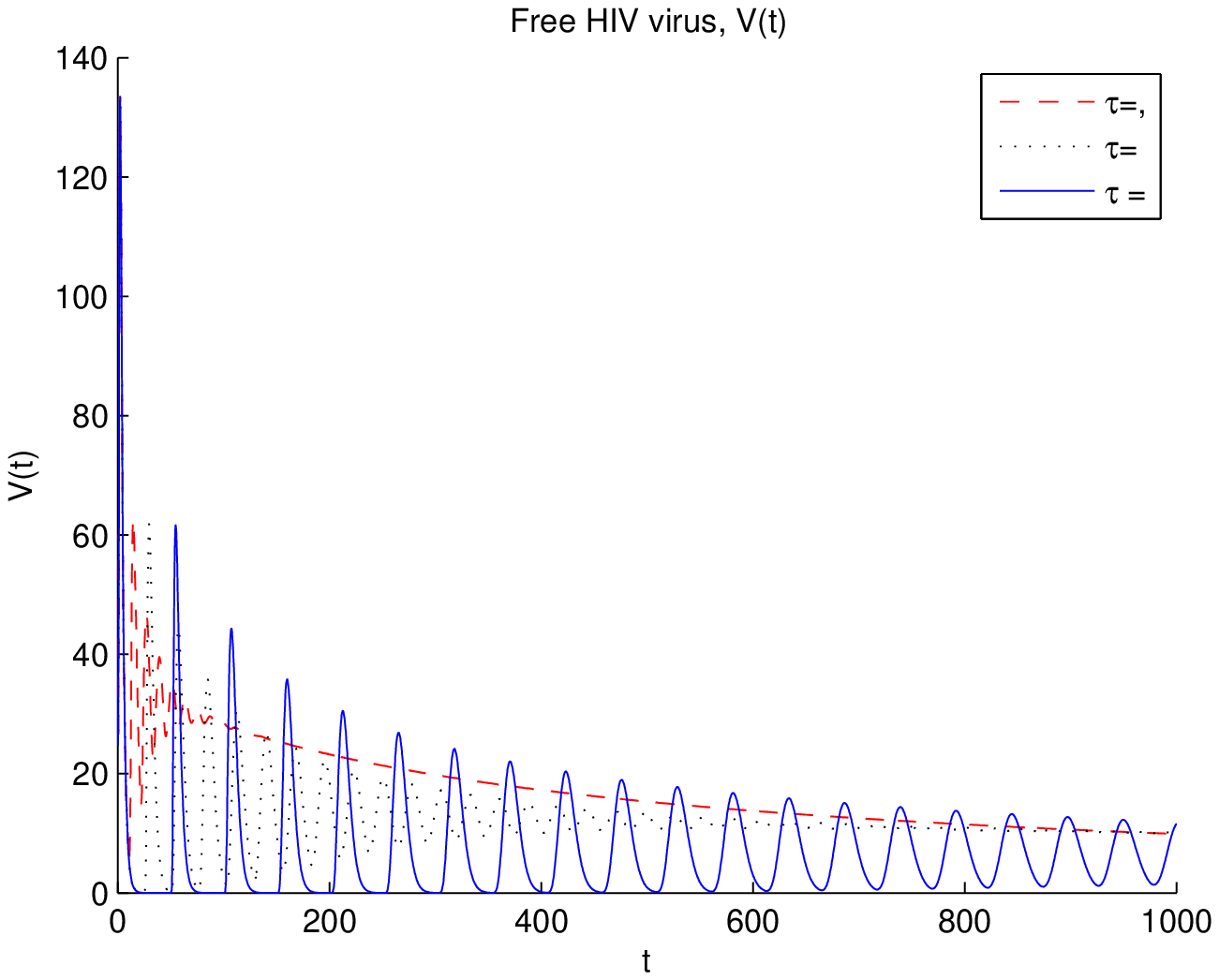}}
\caption{Stability of the infected equilibrium, $E_2$, with several values of $\tau$}
\label{figure3}
\end{figure}

In figure 4, we illustrate the the dynamic for system. \eqref{hiv3} with respect to stability of infected equilibrium for several values of $\tau$.
\begin{figure}[H]
\centering
\subfigure[Dynamics, $T(t)$, $I(t)$, $V(t)$]{\includegraphics[scale=0.3]{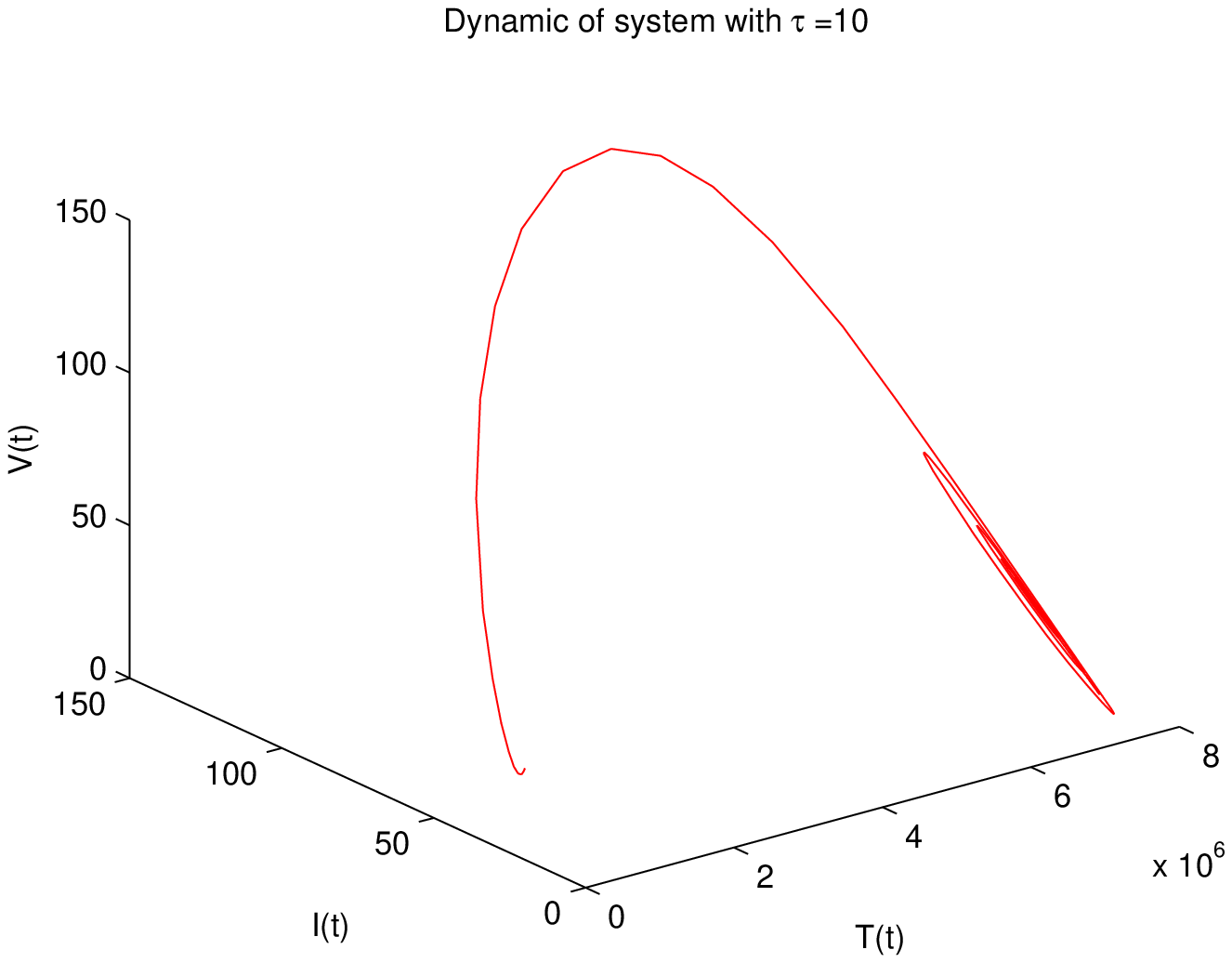}}
\subfigure[Dynamics, $T(t)$, $I(t)$, $V(t)$]{\includegraphics[scale=0.3]{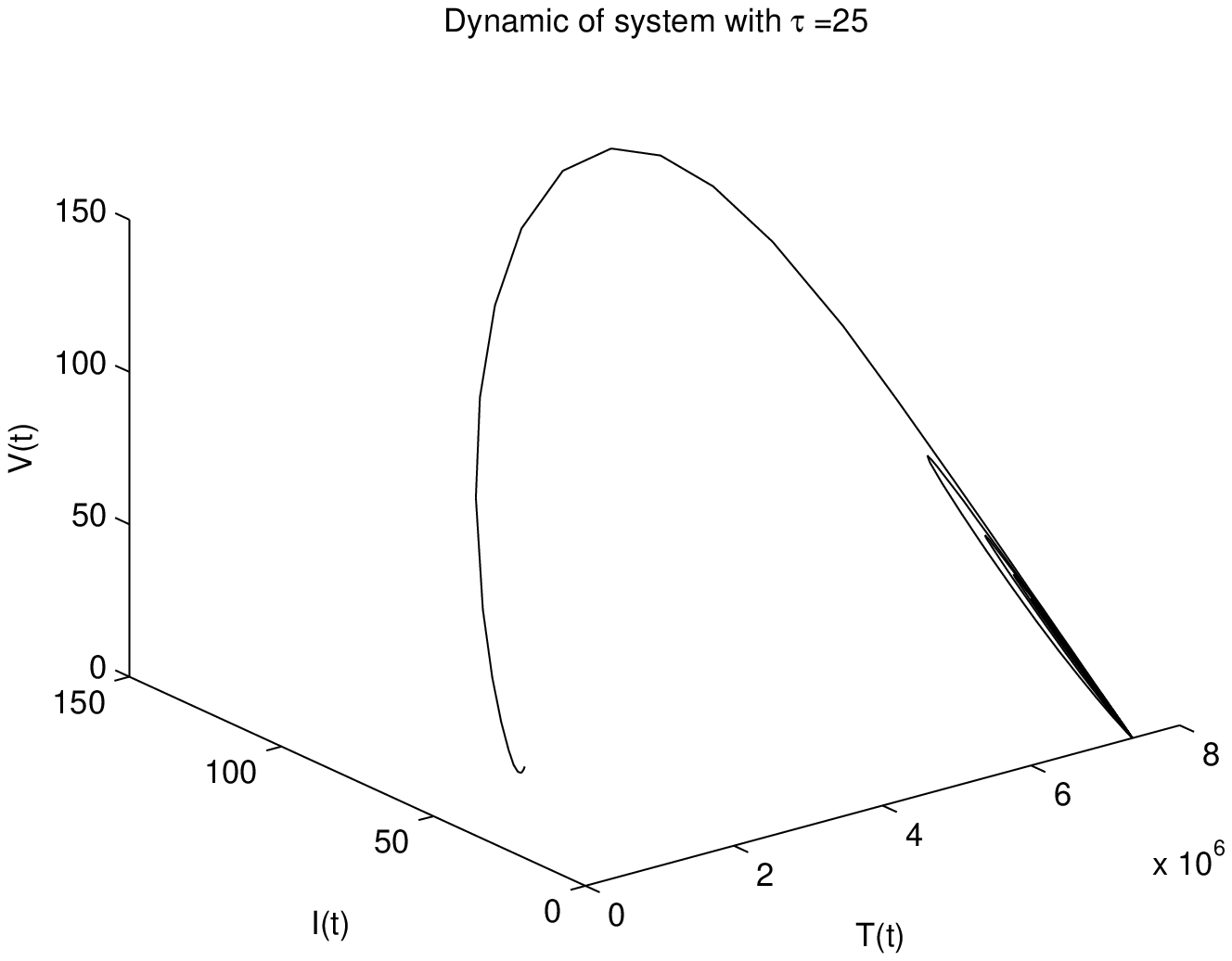}}
\subfigure[Dynamics, $T(t)$, $I(t)$, $V(t)$]{\includegraphics[scale=0.3]{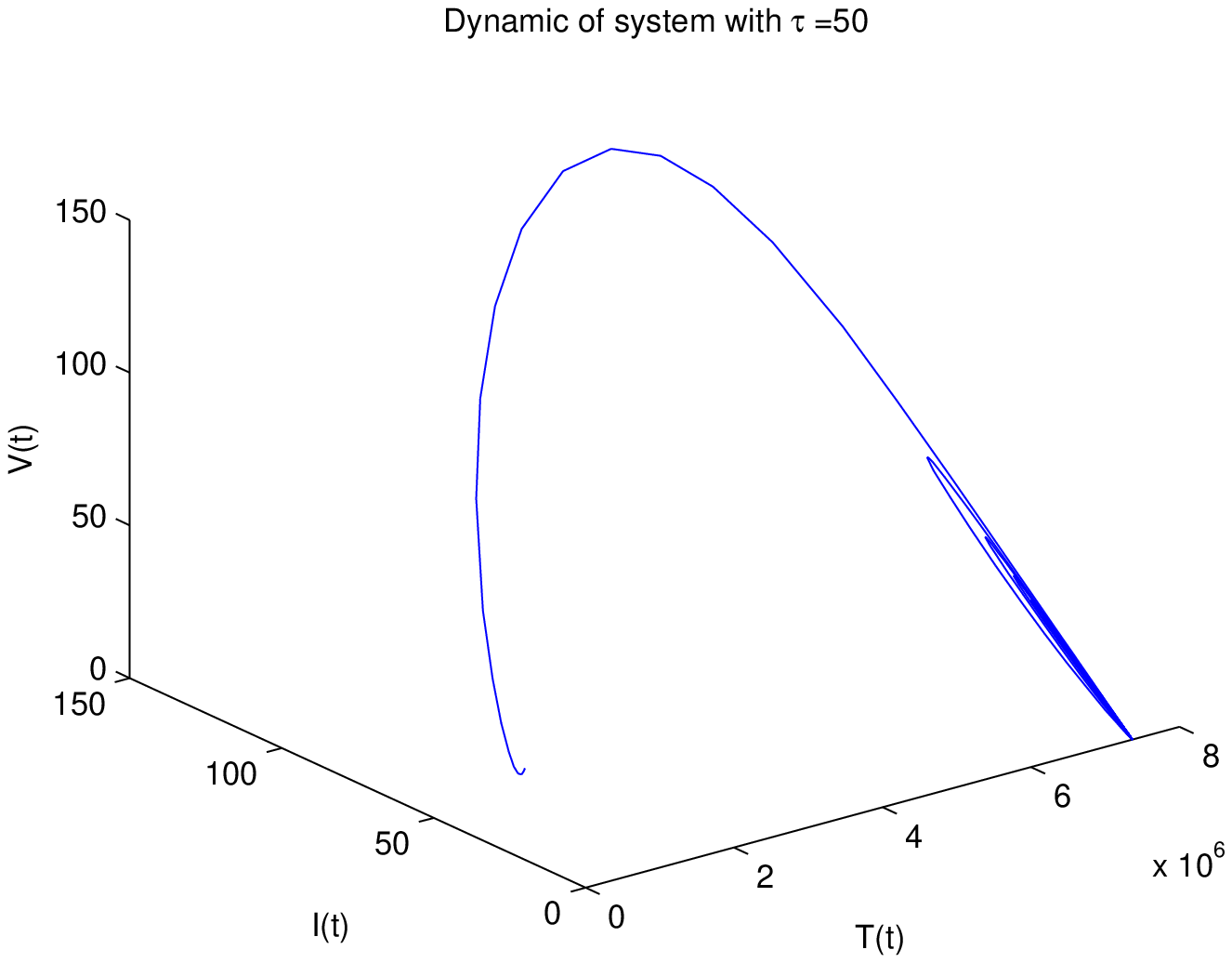}}
\caption{Phase space for different values of $\tau$, which illustrates the stability of infection equilibrium $E_2$}
\label{figure4}
\end{figure}

In figures \ref{figure5}, \ref{figure6} and \ref{figure7}, we illustrate the stability of $E_2$ according to theorem 5. In this case $s=0.01$, $d=0.02$, $a=0.95$, $T_{max}=1200$ $b=0.0027$ $\alpha=0.001$ $\mu=1$, $p=10$, $c=2.4$ with this values $R_0=13.48$ , and we illustrate for different values of $\tau$. The equilibrium is $(19.2, 106.7,444.5)$, we can observe periodic solutions for a large value of $\tau$, also we note that the condition for the global stability of endemic equilibrium is not satisfied.
\begin{figure}[H]
\centering
\subfigure[Healthy, $T(t)$ ]{\includegraphics[scale=0.3]{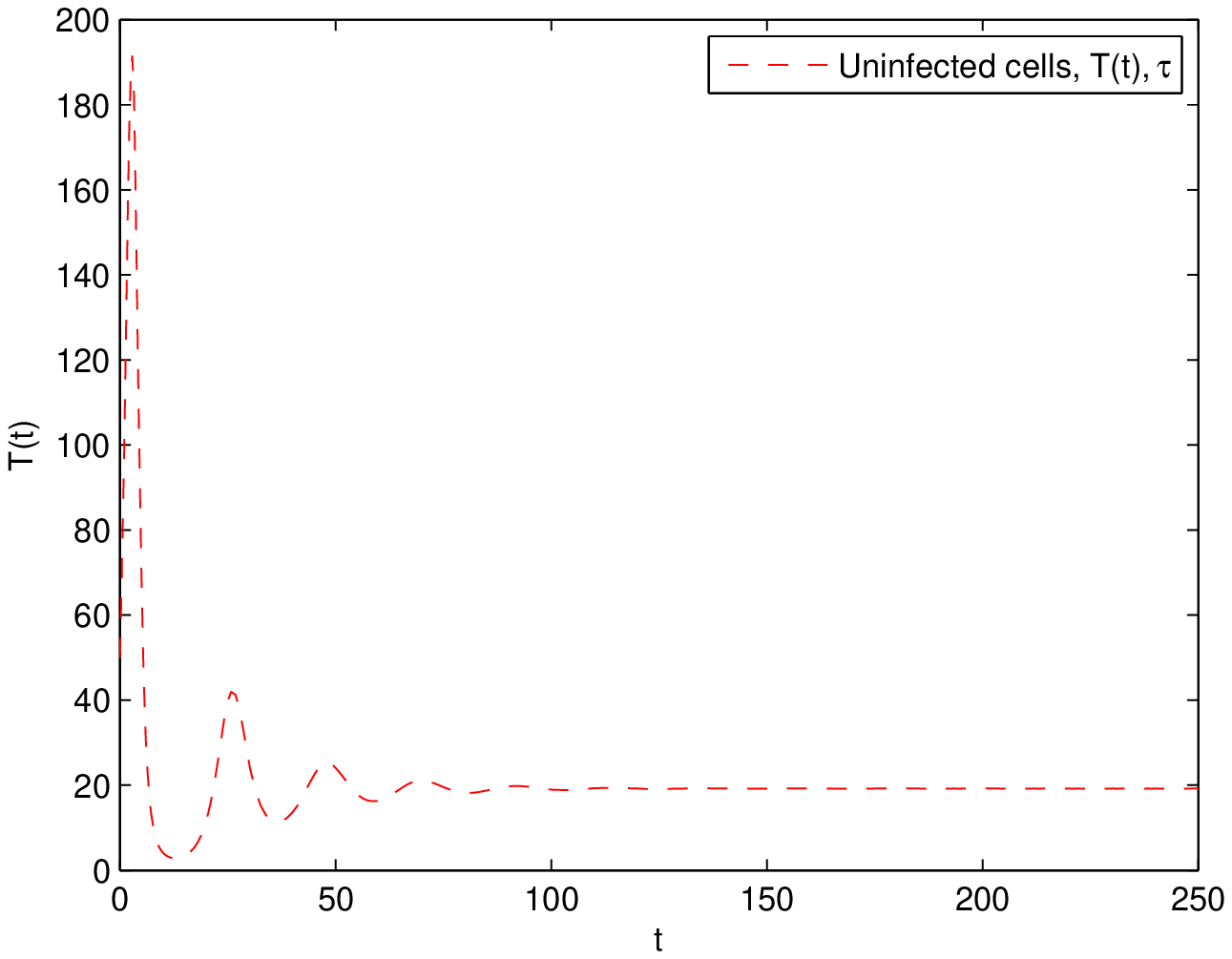}}
\subfigure[Infected, $I(t)$]{\includegraphics[scale=0.3]{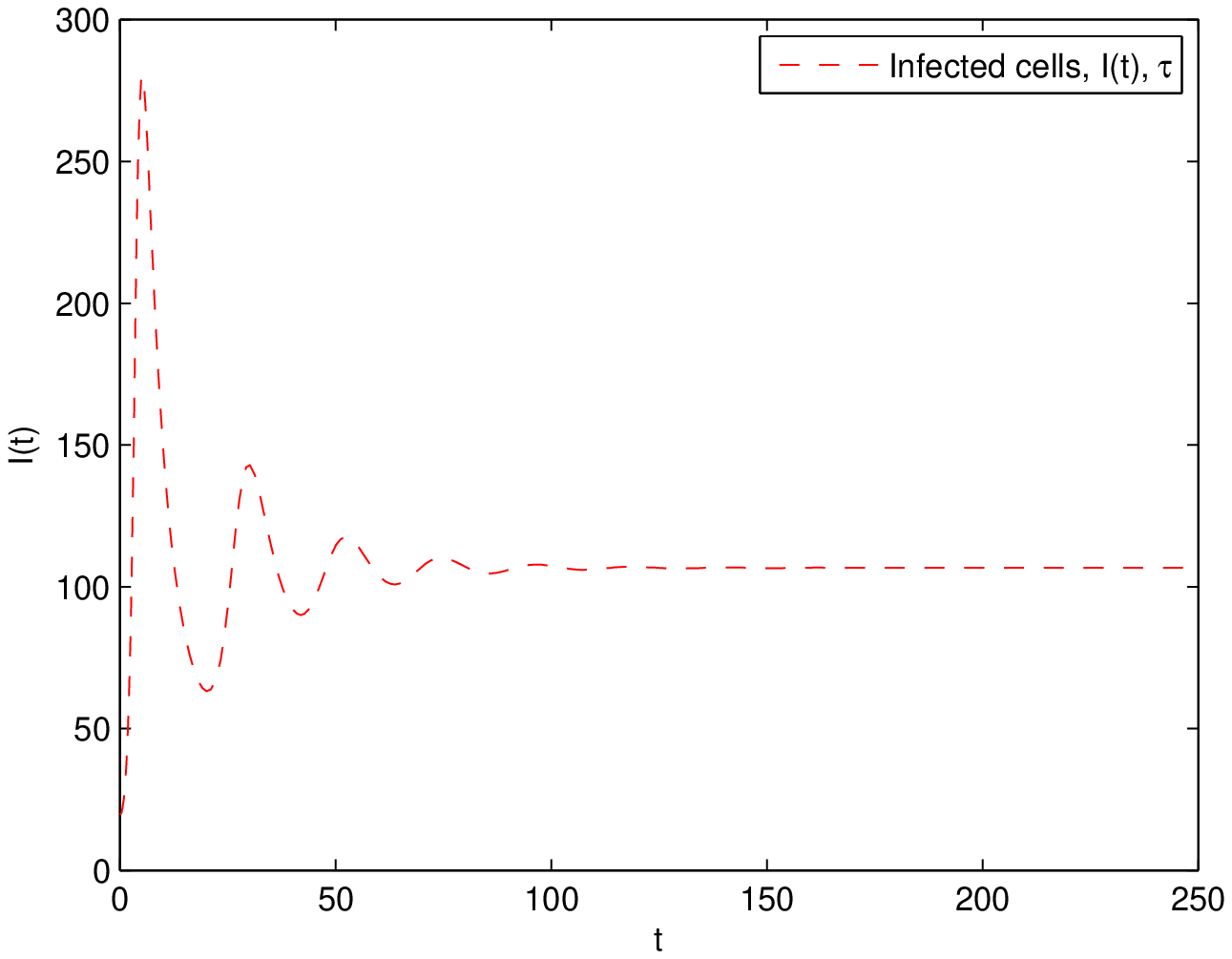}}
\subfigure[Virus, $V(t)$]{\includegraphics[scale=0.3]{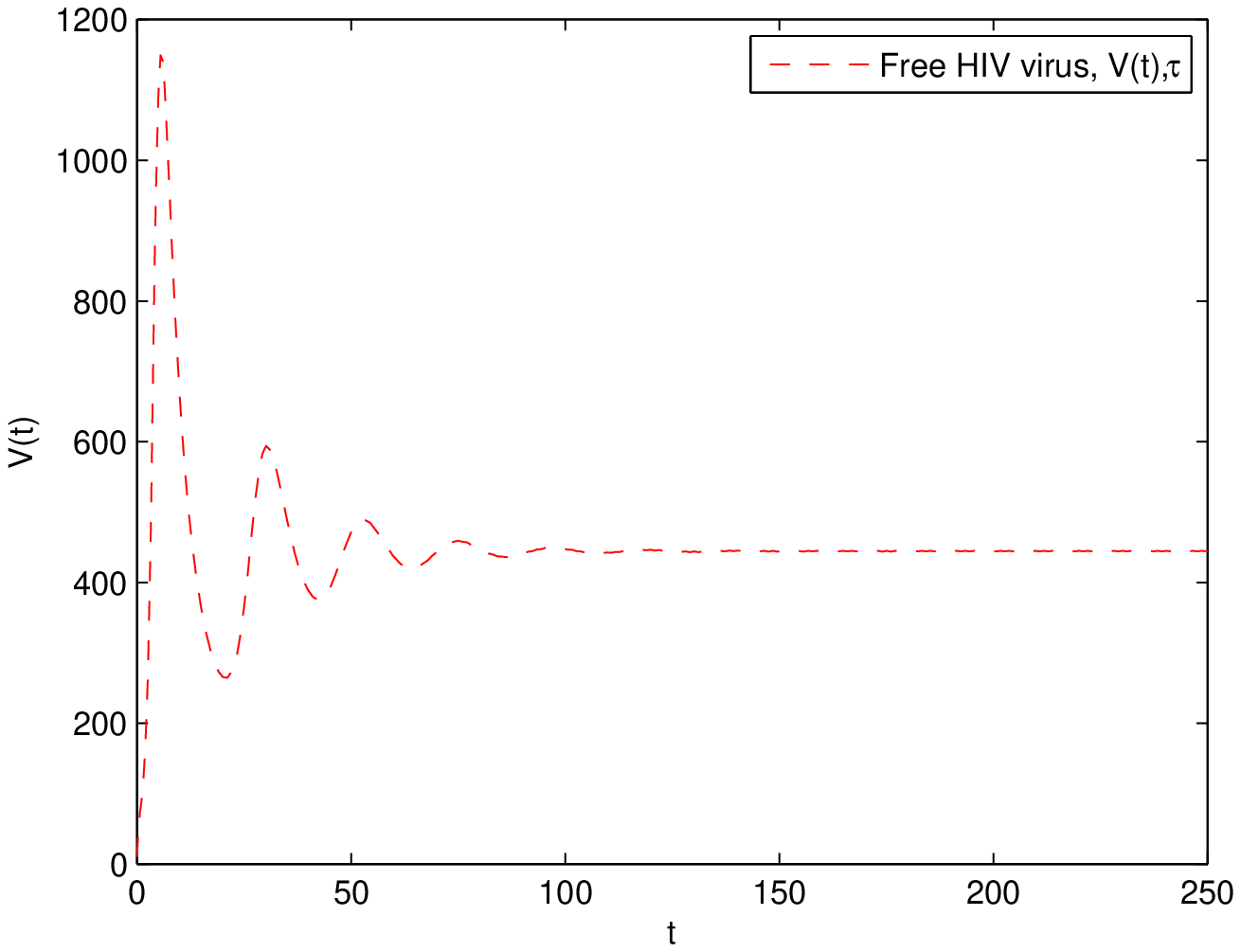}}
\caption{Stability of the $E_2$ equilibrium with $\tau =0.1$}
\label{figure5}
\end{figure}

\begin{figure}[H]
\centering
\subfigure[Healthy, $T(t)$ ]{\includegraphics[scale=0.3]{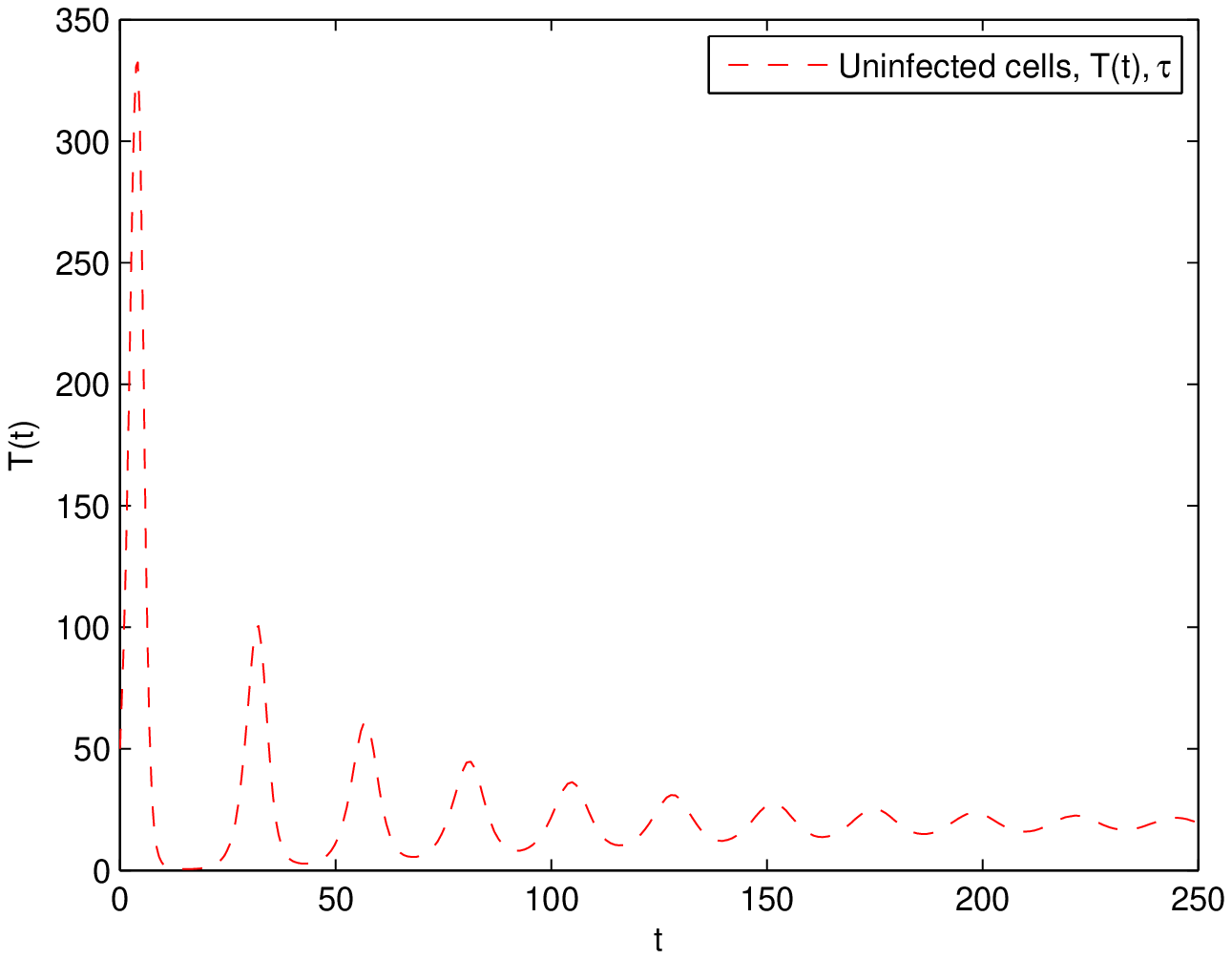}}
\subfigure[Infected, $I(t)$]{\includegraphics[scale=0.3]{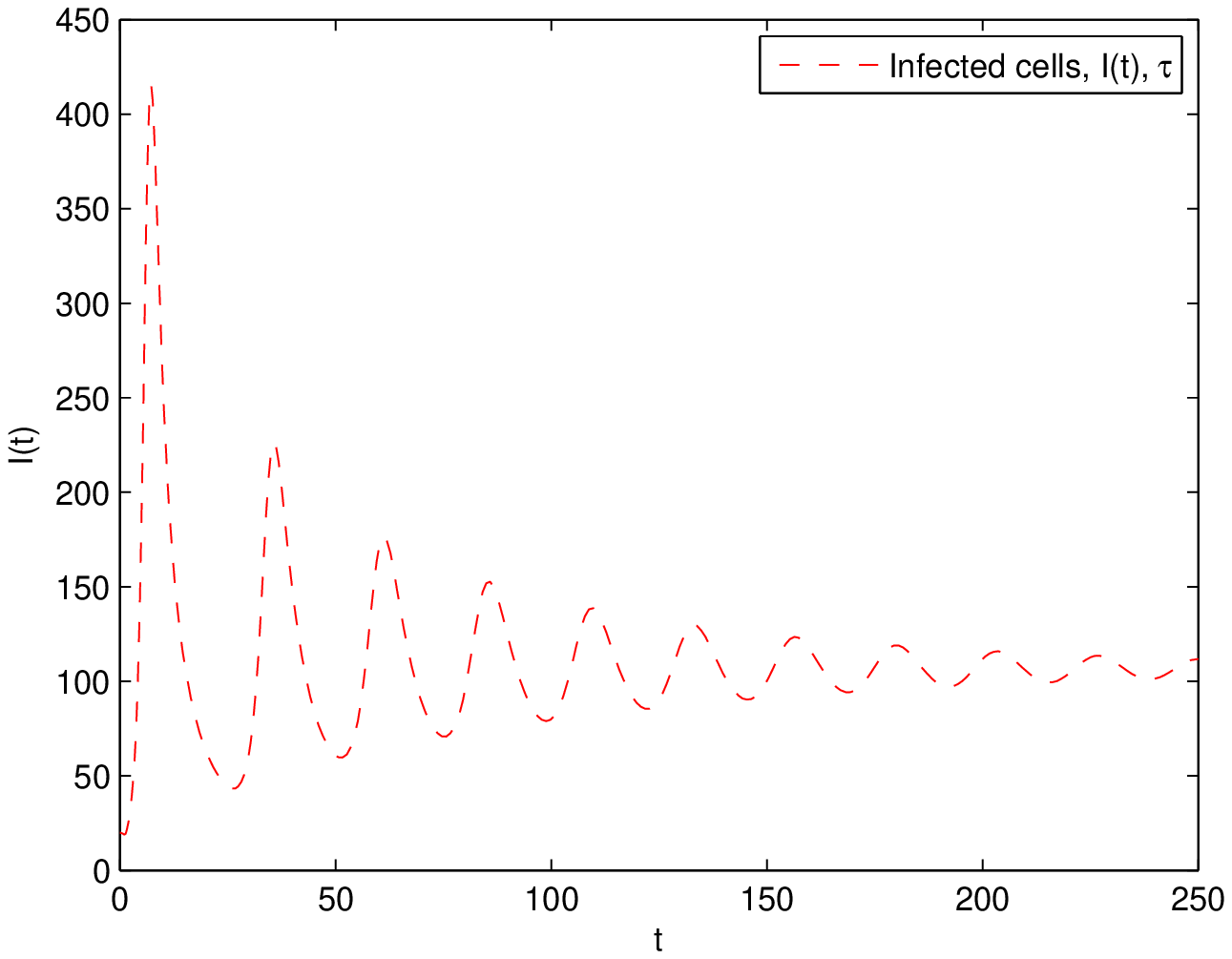}}
\subfigure[Virus, $V(t)$]{\includegraphics[scale=0.3]{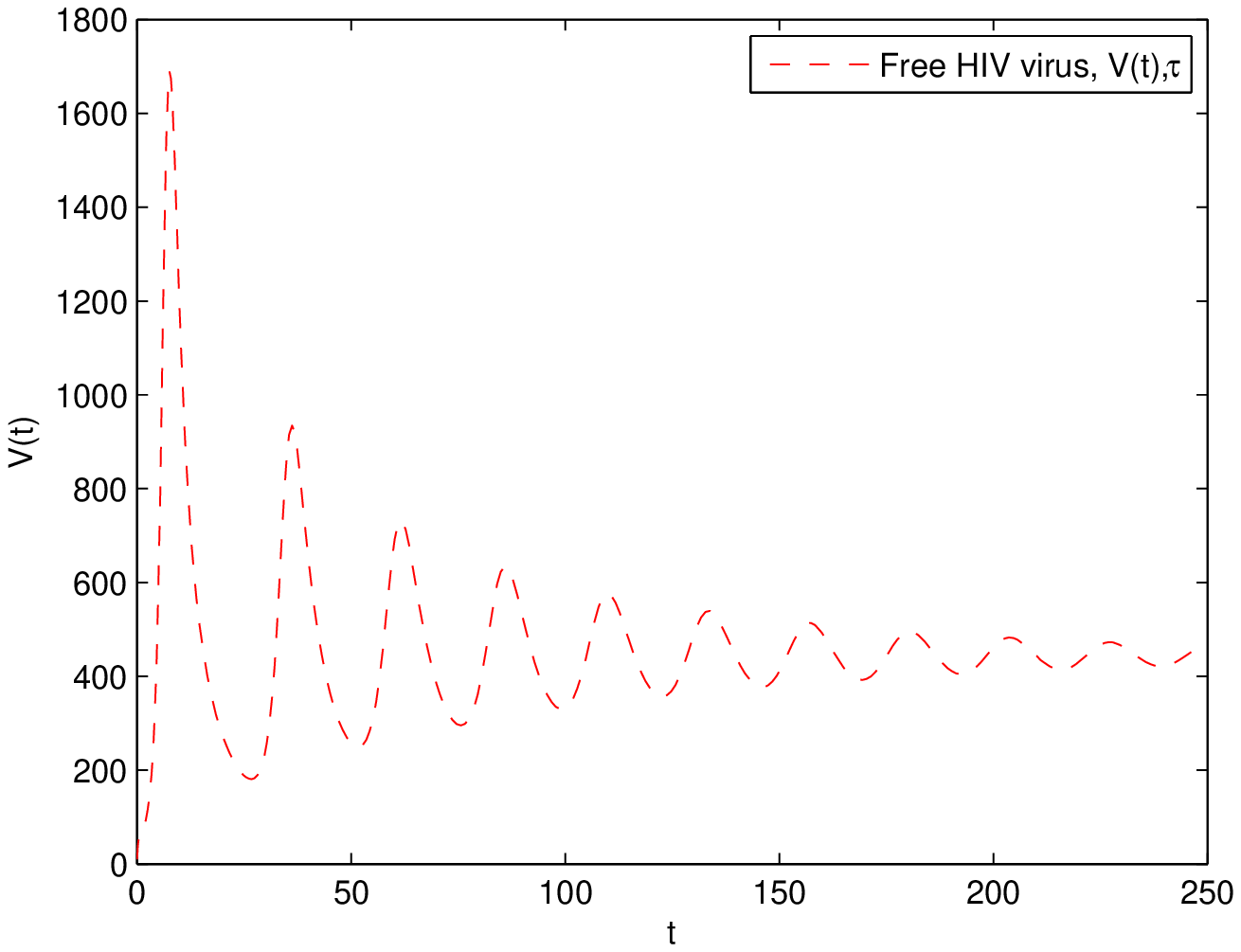}}
\caption{Stability of the $E_2$ equilibrium with $\tau =1$}
\label{figure6}
\end{figure}

\begin{figure}[H]
\centering
\subfigure[Healthy, $T(t)$ ]{\includegraphics[scale=0.3]{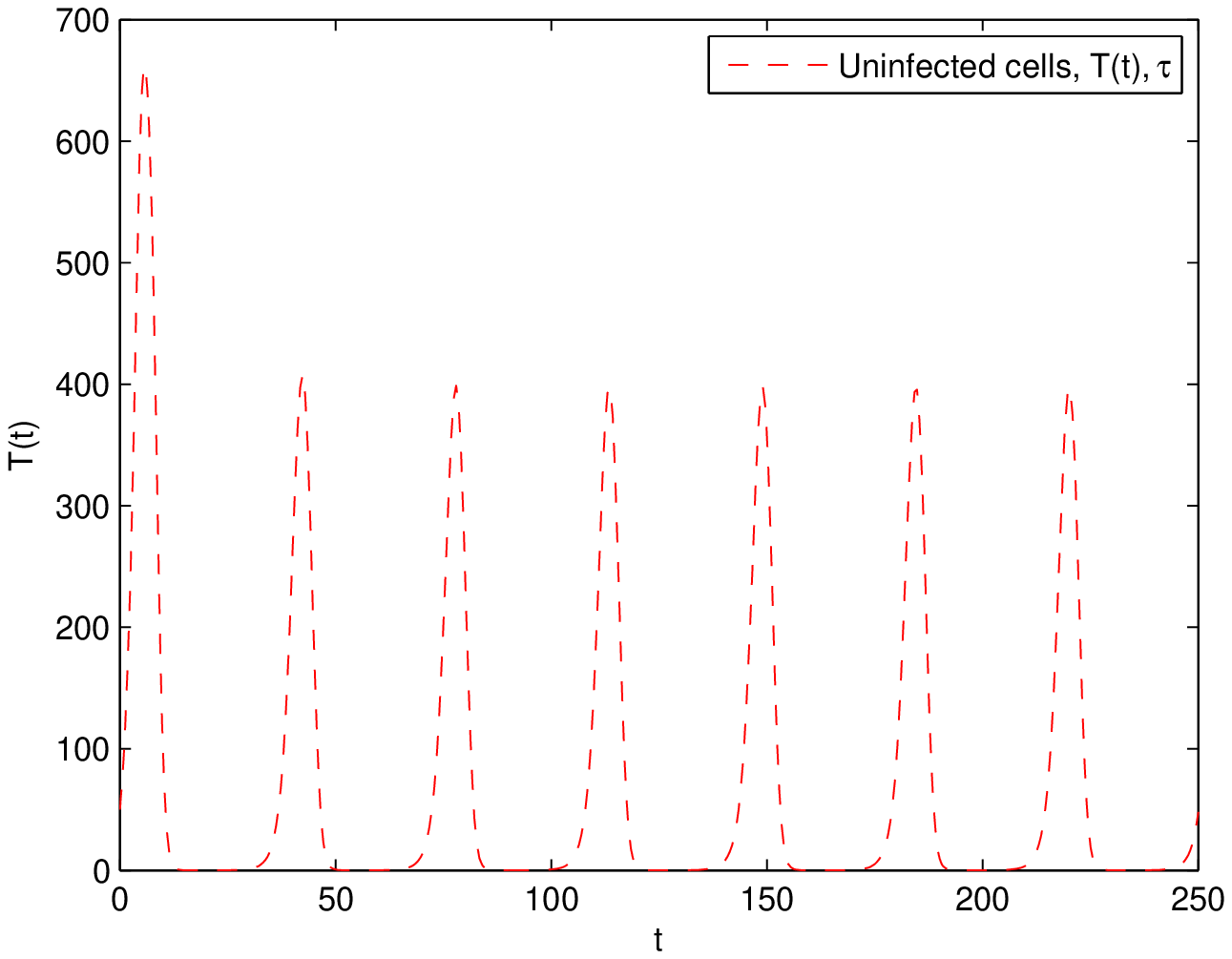}}
\subfigure[Infected, $I(t)$]{\includegraphics[scale=0.3]{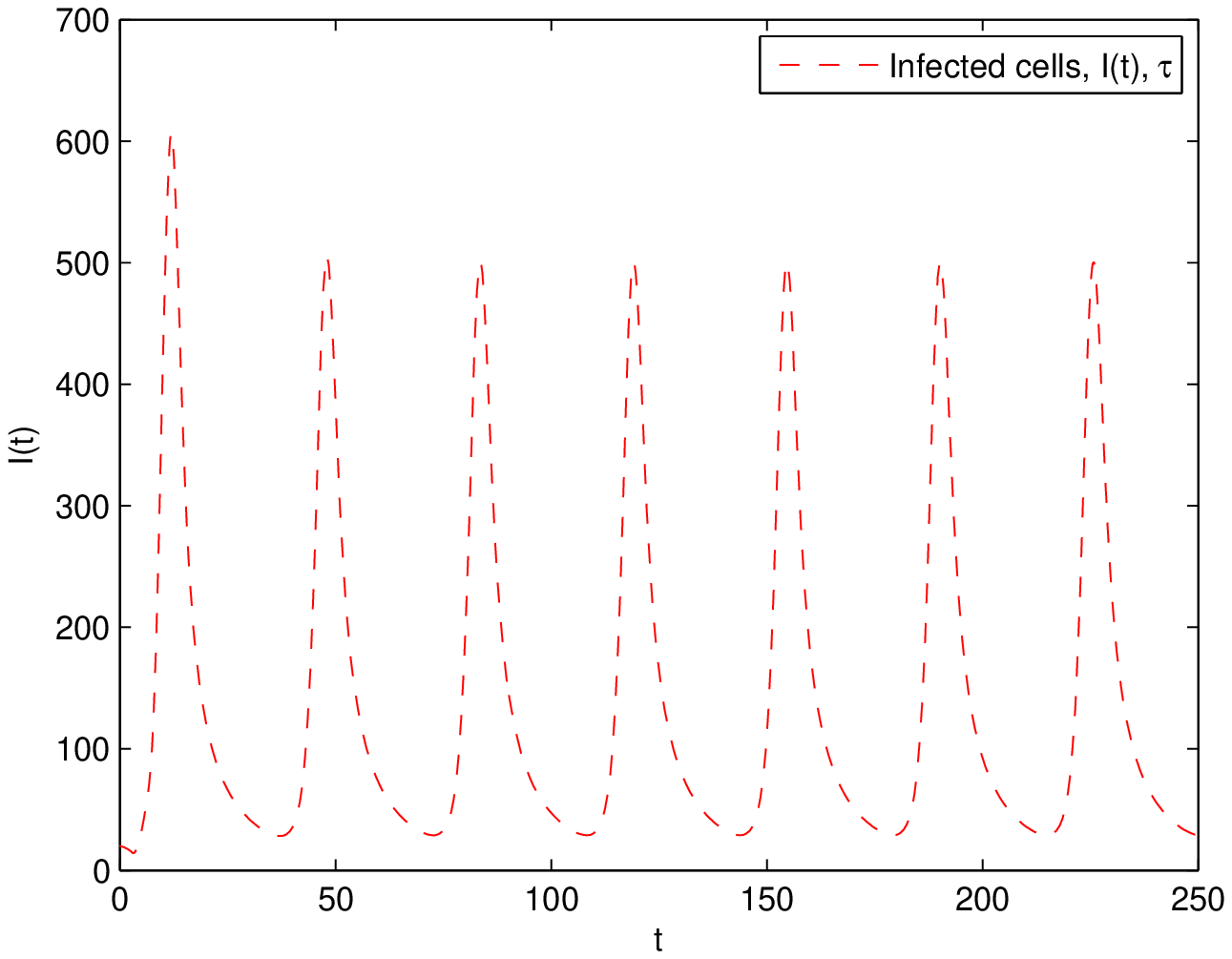}}
\subfigure[Virus, $V(t)$]{\includegraphics[scale=0.3]{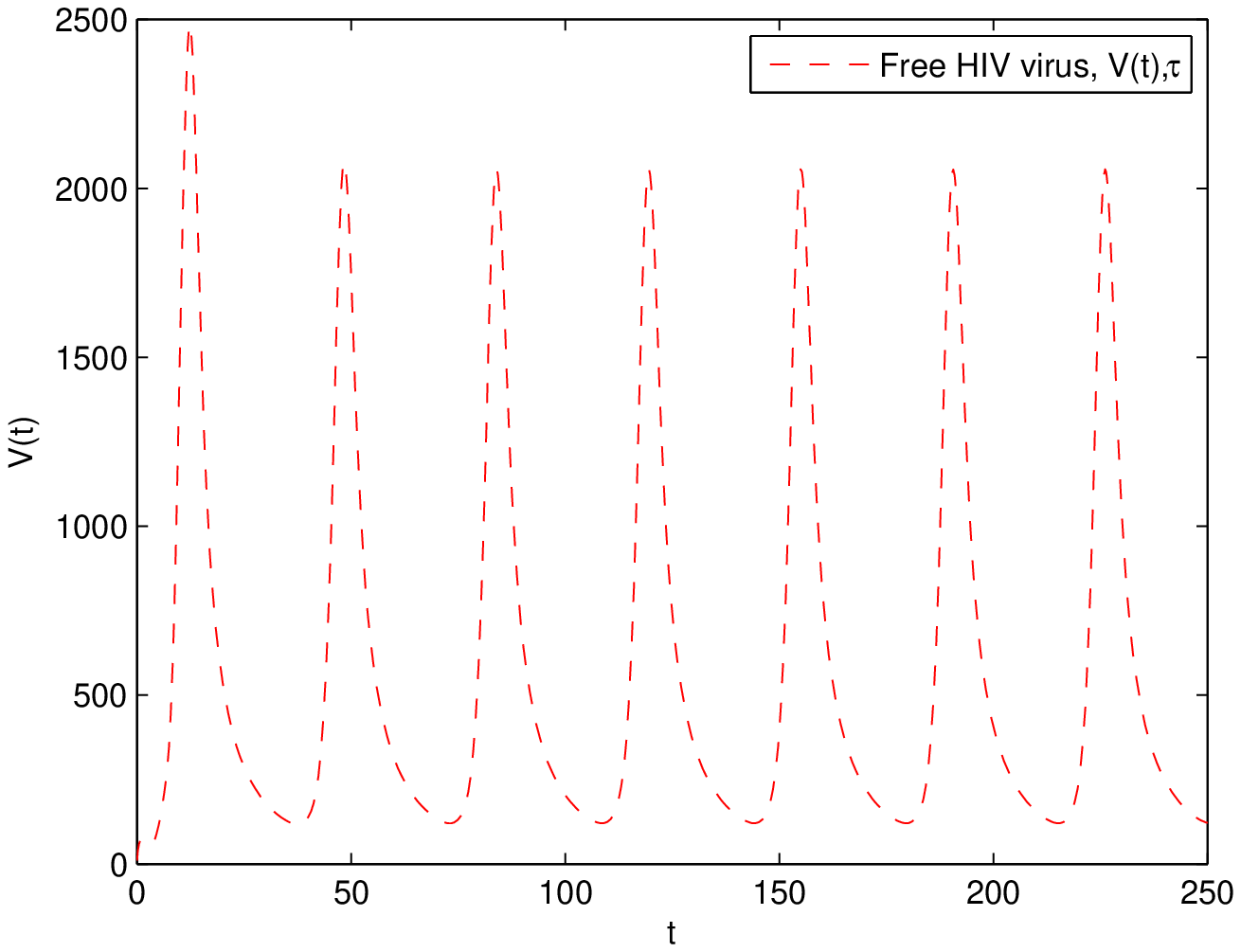}}
\caption{Stability of the $E_2$ equilibrium with $\tau =3$}
\label{figure7}
\end{figure}

In figure \ref{figure7C} we take increasing values for $c$, the declination rate of virions, for the simulations we consider the values for the parameters as $s=0.01$, $d=0.02$, $a=0.95$, $T_{max}=1200$ $b=0.0027$ $\alpha=0.001$ $\mu=1$, $p=10$, and $\tau =4$ we can observe that for a increasing value of $c=13$ the periodic solutions have a smaller period and they disappear for $c$ large enough
\begin{figure}[H]
\centering
\subfigure[Periodic solution for $c=5$ ]{\includegraphics[scale=0.3]{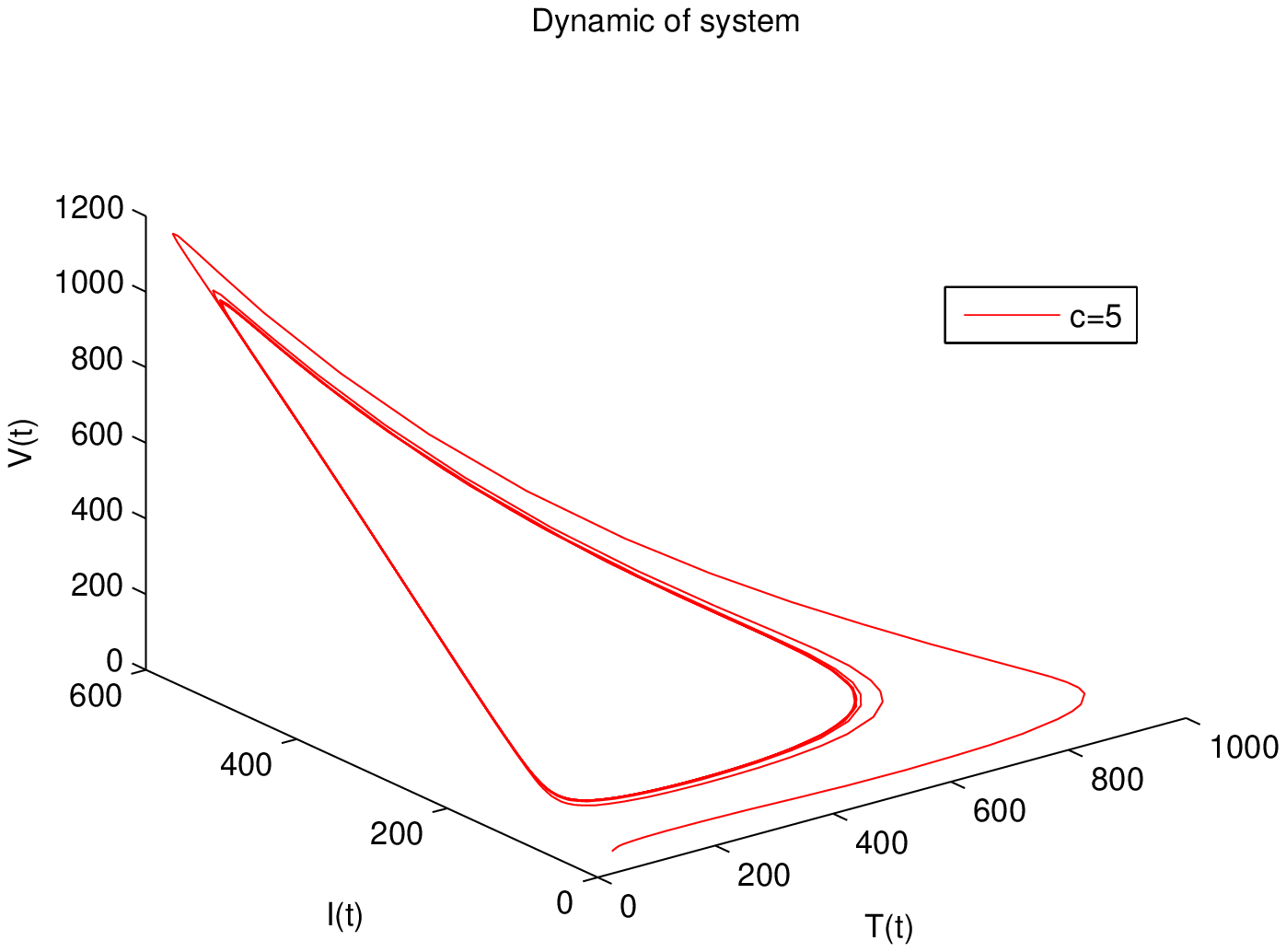}}
\subfigure[Periodic solution for $c=10$ ]{\includegraphics[scale=0.3]{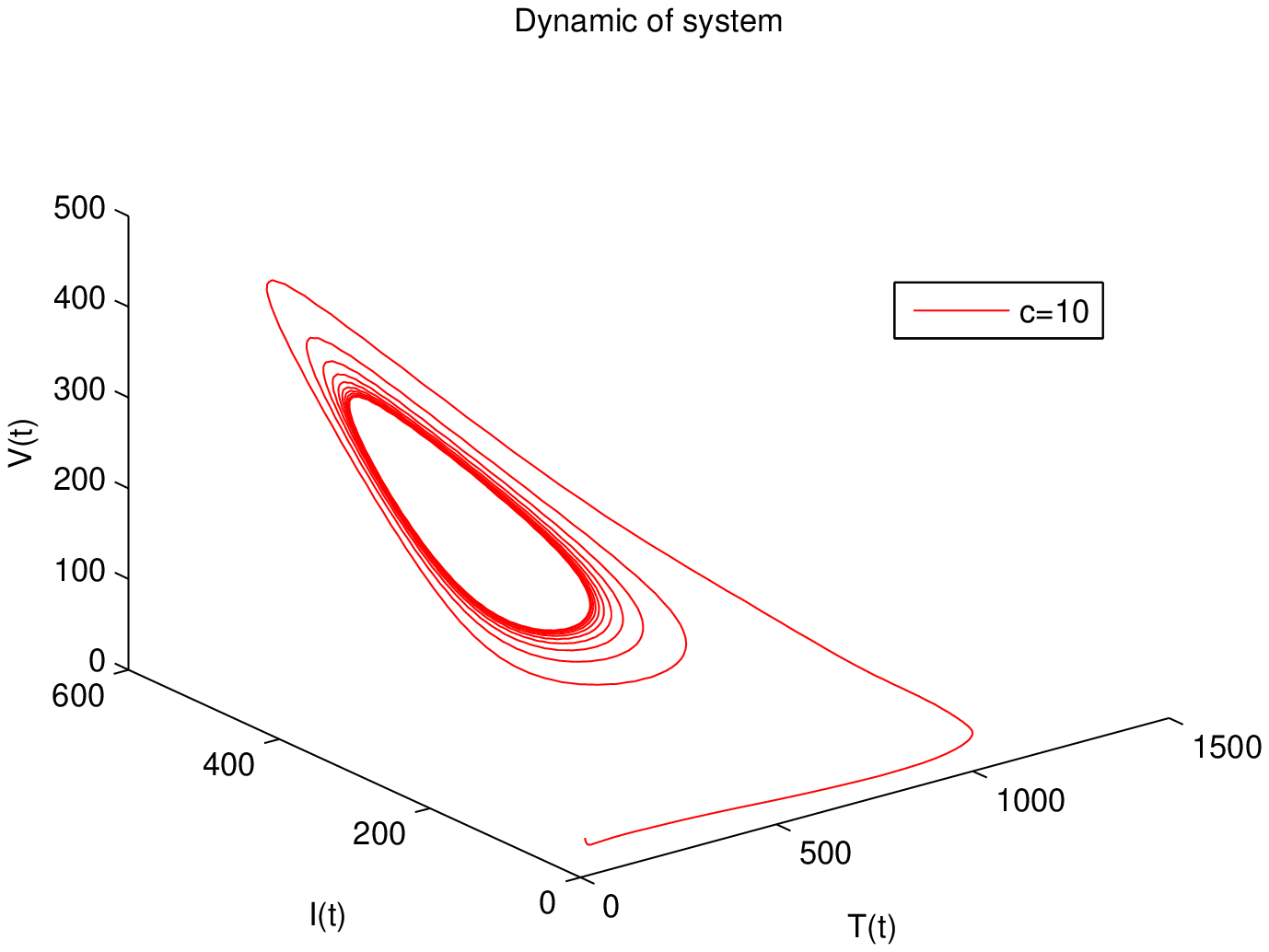}}
\subfigure[Periodic solution for $c=13$ ]{\includegraphics[scale=0.3]{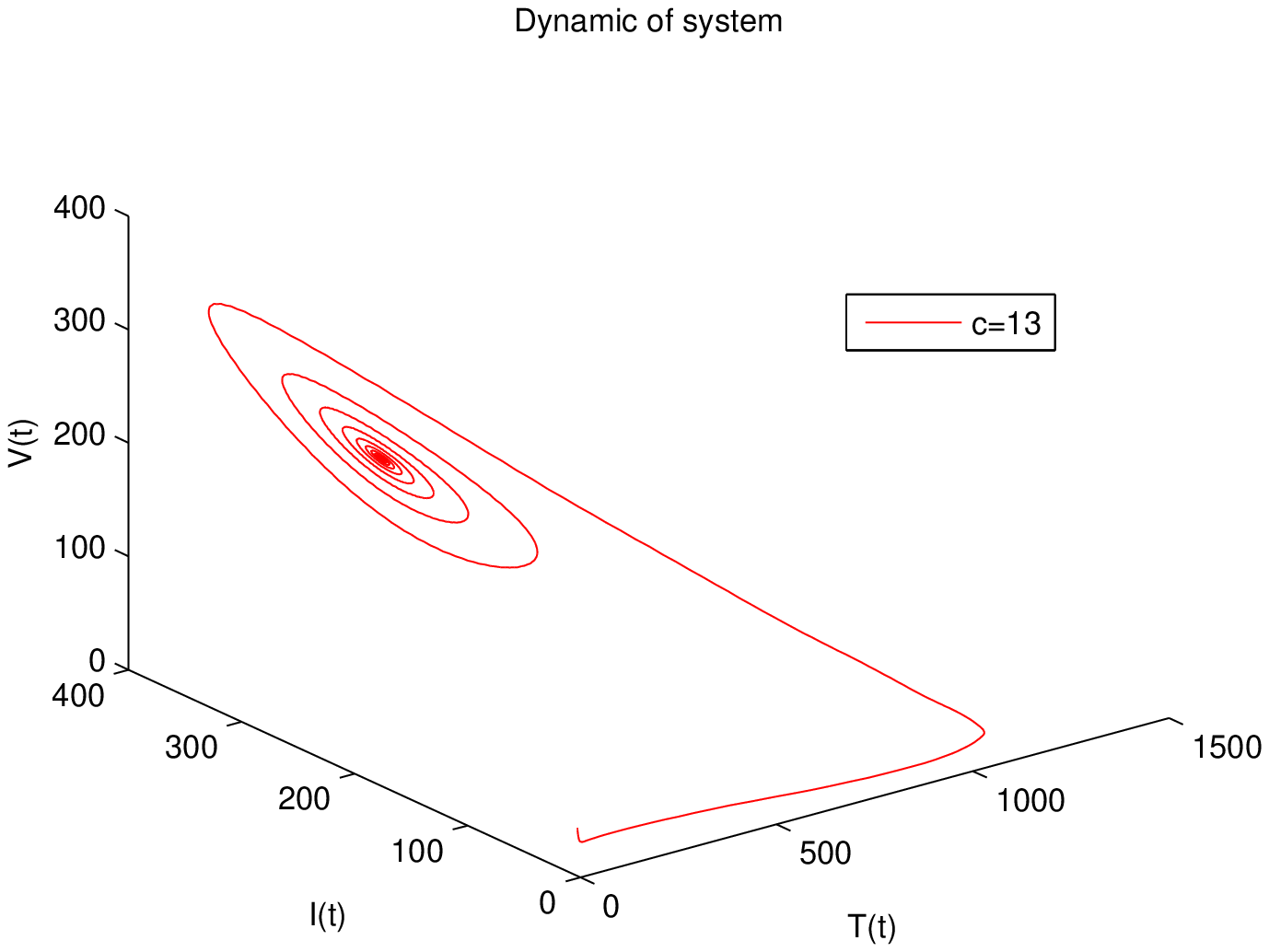}}
\caption{Periodic solutions for different values of $c$}
\label{figure7C}
\end{figure}
 The numerical simulations  shows that if the clearance rate of viral particles  $c$ is sufficiently large (small) then nonexistence (existence) periodic orbits.

Now taking $\alpha =0.005$ we can see the effect of saturation in this case the equilibrium is $(314.2, 278.7,1161.2)$ and the dynamics for different values of $\tau$ are in figure \ref{figure8}, again we can observe periodic solutions for system.
\begin{figure}[H]
\centering
\subfigure[Dynamics, $T(t)$, $I(t)$, $V(t)$, with $\tau =10$]{\includegraphics[scale=0.3]{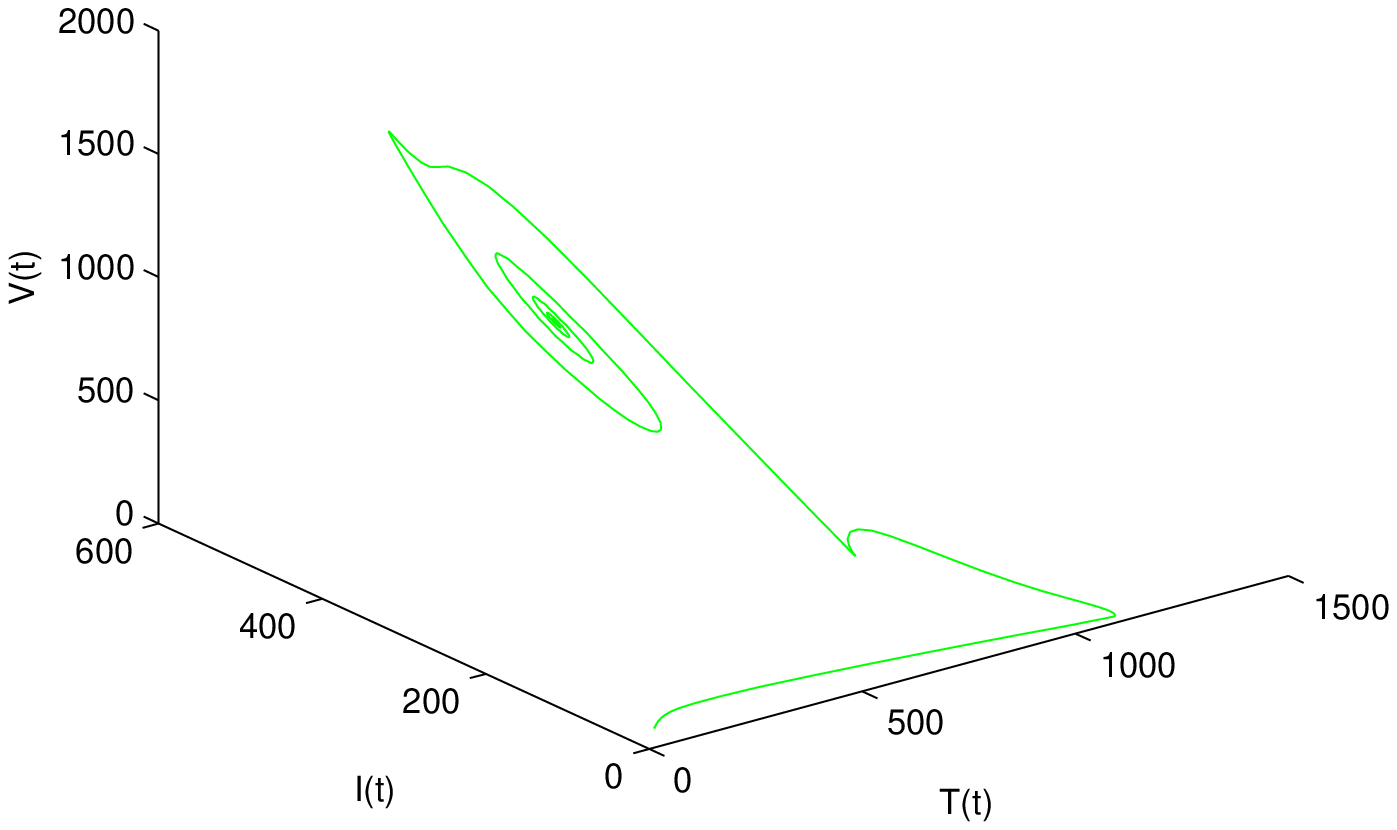}}
\subfigure[Dynamics, $T(t)$, $I(t)$, $V(t)$, with $\tau =15$]{\includegraphics[scale=0.3]{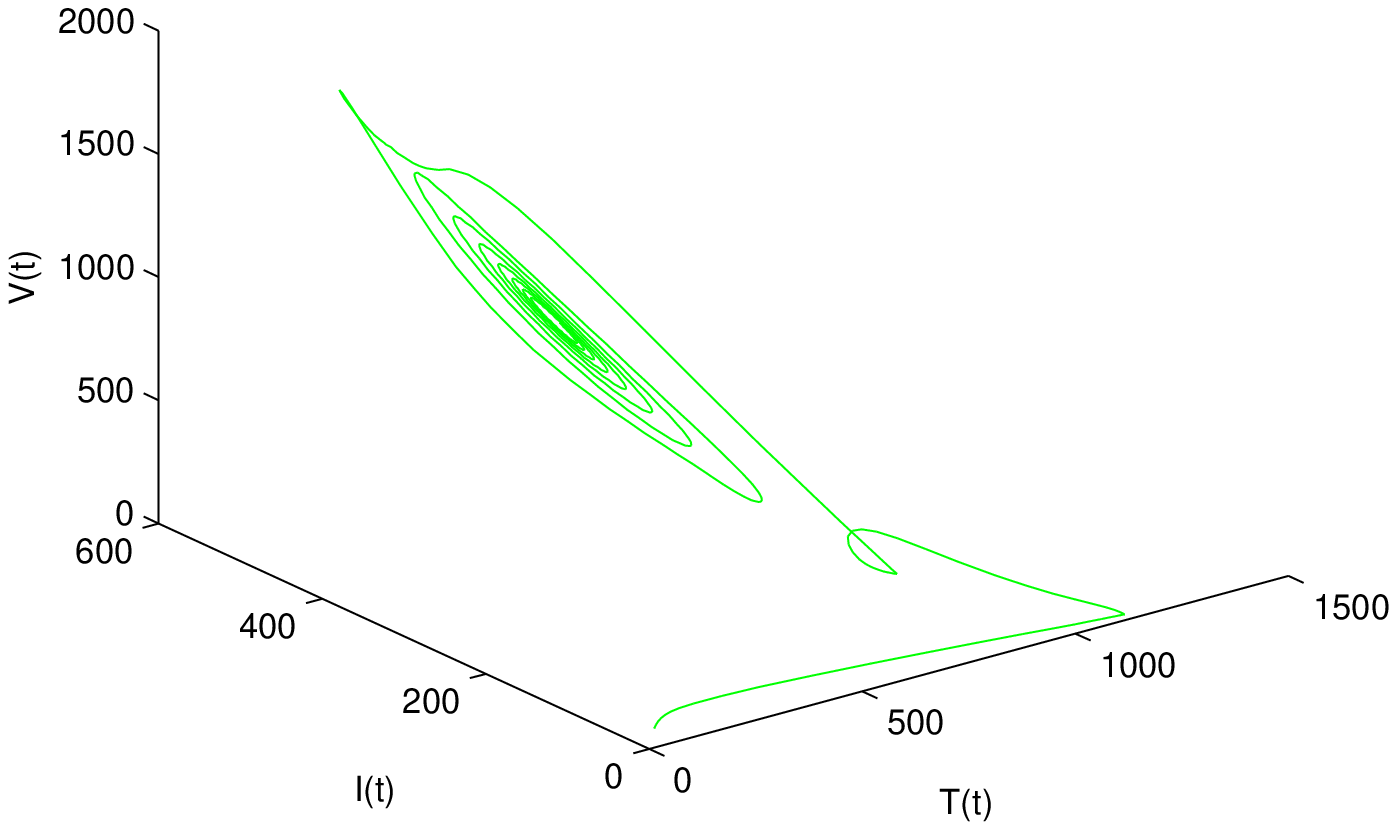}}
\subfigure[Dynamics, $T(t)$, $I(t)$, $V(t)$, with $\tau =25$]{\includegraphics[scale=0.3]{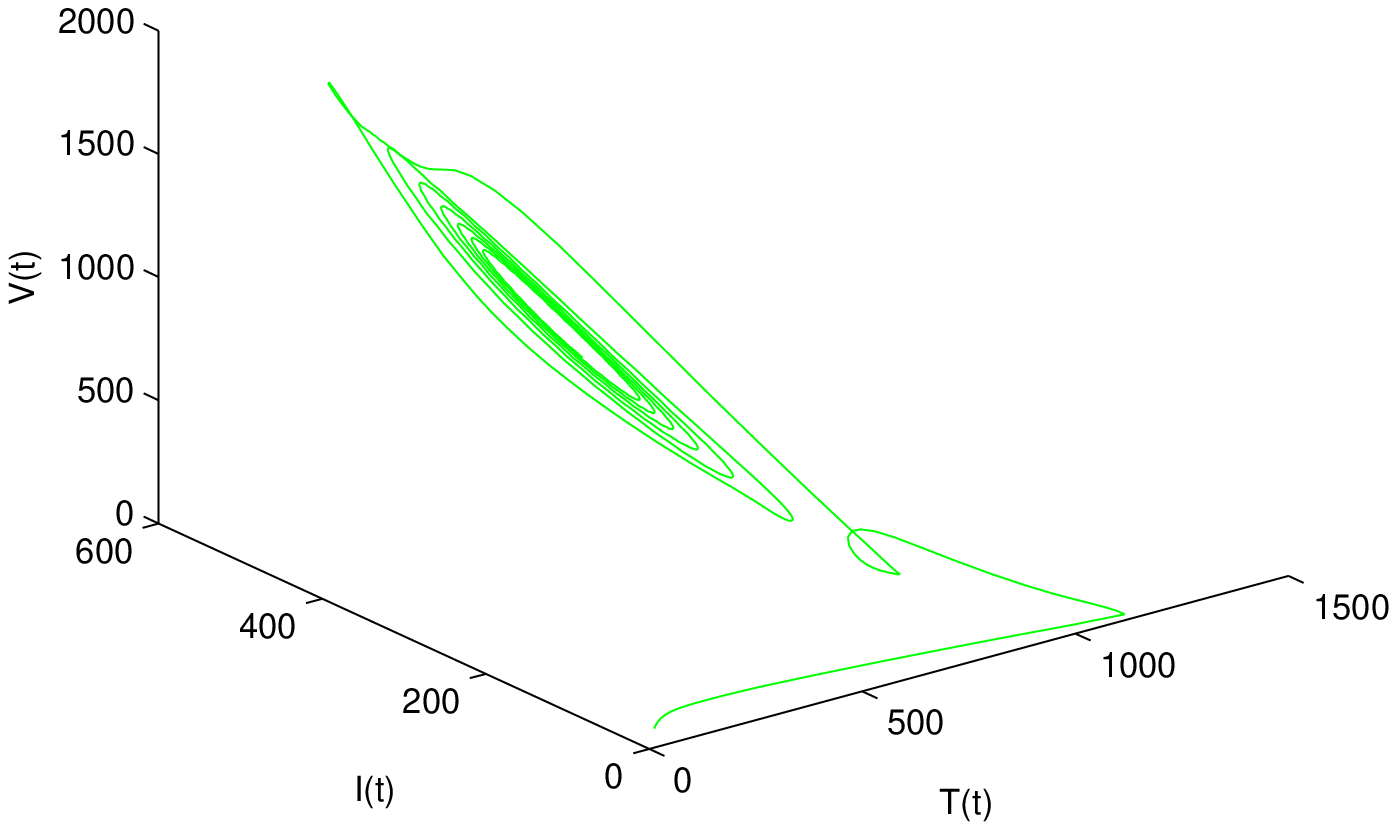}}
\subfigure[Dynamics, $T(t)$, $I(t)$, $V(t)$, with $\tau =35$]{\includegraphics[scale=0.3]{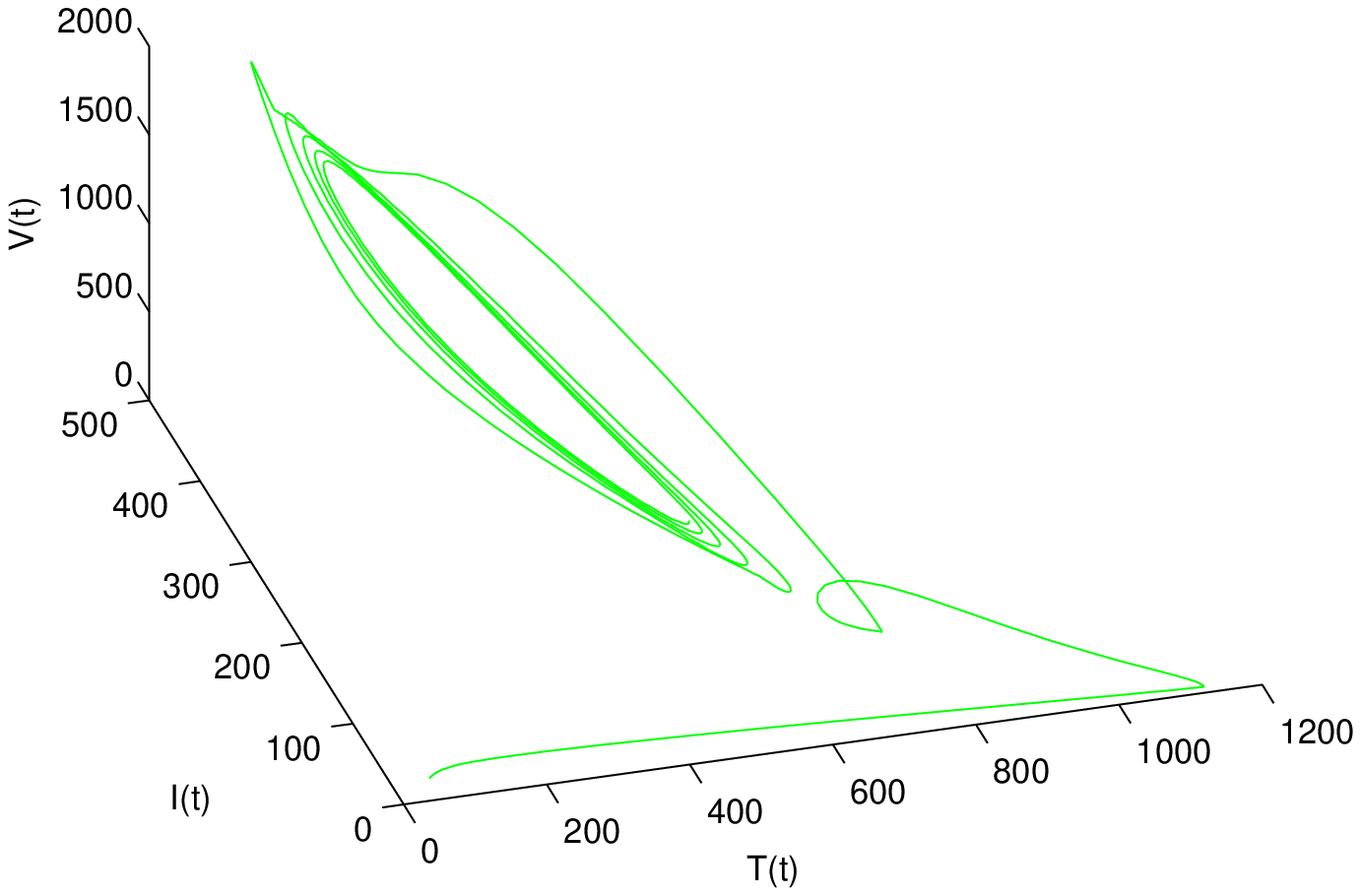}}
\caption{Dynamic of system with $\alpha =0.005$}
\label{figure8}
\end{figure}

\section{Sensitivity analysis}\label{sensitivity}

In this section we provide a local sensitivity analysis of the basic reproduction number, in order to assess which parameter has the greatest influence on changes of $R_0$ values and hence the greatest effect in determining whether the disease will be cleared in the population (see e.g. \cite{chetal}).\\
To this aim, denoting by $\Psi$ the generic parameter of system \eqref{hiv3}, we evaluate the \emph{normalised sensitivity index}
$$
S_{\Psi}=\frac{\Psi}{R_0}\frac{\partial R_0}{\partial \Psi}
$$
which indicates how sensitive $R_0$ is to a change of parameter $\Psi$. A positive (resp. negative) index indicates that an increase in the parameter value results in an increase (resp. decrease) in the $R_0$ value.

We consider the values $s=0.01$, $d=0.02$, $a=0.95$, $T_{max}=1200$ $b=0.0027$ $\alpha=0.001$ $\mu=1$, $p=10$, $c=2.4$, in order to evaluate the normalised sensitivity index, we are able to show our results in figure \ref{figure9}, in the figure we can appreciate that the parameters $b$, $p$ the infection constant rate and the reproductive rate respectively have a positive influence in the value of $R_0$ which means if we increase or decrease any of them by say $10\%$ then $R_0$ will increase or decrease by $9.9\%$, note that $T_{\max}$ have the same effect. The index for parameters $\mu$ and $c$ which represent the death rate of infected cells and the rate of clean of virions, show that increasing this values say by $10\%$ will decrease the value of $R_0$ almost by a $10\%$. While for the parameters $a$, $d$ an increase of $10\%$ of their value will increase and decrease respectively $R_0$ by $2\%$ and $1.9\%$ and to increase $s$ by a $100\%$ will increase the value of $R_0$ by $0.8\%$.
\begin{figure}[H]
\centering
\includegraphics[scale=0.8]{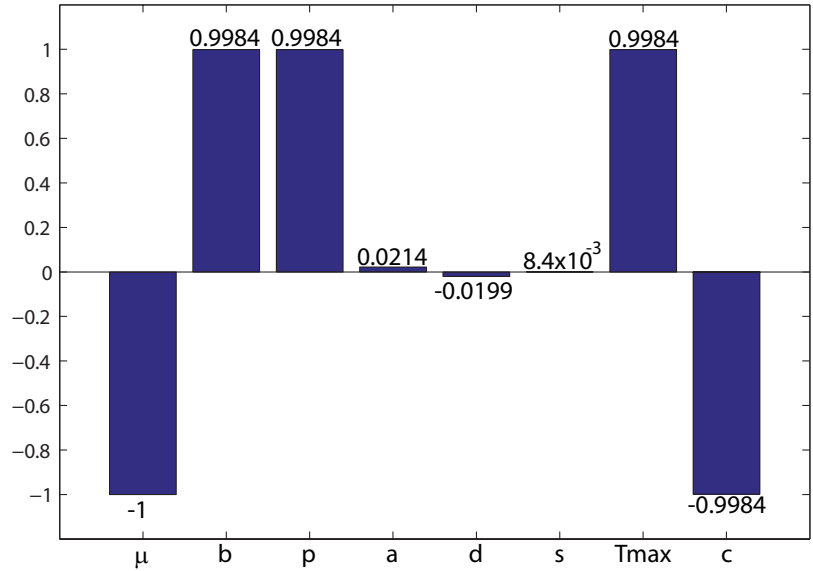}
\caption{Sensitivity index for $R_0$ with respect to some chosen parameters}
\label{figure9}
\end{figure}
Other effects over $R_0$ can be analyzed if we consider other values for the parameters.
\section{Conclusions}

In this paper, we extend the hepatitis model with mitotic
transmission to incorporate the effect of saturation infection
function and an intracellular time delay between infection of a
infected target cell and production of new virus particles.

We have established results about the local and global stability of equilibria. We can conclude than the infection-free stability is completely determined by the value of the basic reproductive number $R_0$, if $R_0\leq 1$ then the infection-free equilibrium will be stable and unstable if $R_0>1$. For the infected equilibrium we established conditions to ensure the local stability. We need the condition $a\leq d+\dfrac{a}{T_{\max}}[T_2+I_2]$ to ensure the global stability for this equilibrium. We also established conditions for the occurrence of a Hopf Bifurcation. And we established conditions to ensure the permanence of our system. Moreover we made an estimation for the length of delay to preserve stability depending on the parameters of system \eqref{hiv3}.\\

For a non--cytopathic virus ($d=\mu$), we found the sufficient and necessary conditions of the global stability for the infected equilibrium state of hepatitis infection model (\ref{hiv3}). We note that
\begin{eqnarray}\label{sumTI}
T_2+I_2=T_0=\frac{T_{max}}{2a}\left[(a-d)+\sqrt{(a-d)^{2}+\frac{4as}{
T_{max}}}\right].
\end{eqnarray}
Substituting the relation (\ref{sumTI}) into condition $a\leq d+\frac{a}{T_{\max}}[T_2+I_2]$, we obtain
\begin{eqnarray*}
(a-d)\leq\sqrt{(a-d)^{2}+\frac{4as}{ T_{max}}}.
\end{eqnarray*}
It is clear that the inequality is satisfied for all positive parameter values. From Theorem \ref{GASinfected}, we obtain the following corollary.
\begin{corollary} Assume that $d=\mu$ and $R_{0}>1$, then the infected equilibrium  state $E_2$ of  model (\ref{hiv3}) is globally asymptotically stable for any $\tau\geq0$.
\end{corollary}

Our analysis shows that if the infection is not cytopathic
($d=\mu$) then nonexistence periodic solutions.

\section*{Acknowledgements}
This article was supported by UADY under grant FMAT-2012-0004 and Mexican SNI under grants 15284 and 33365.

\end{document}